\documentclass[english,a4paper,11pt]{article}

\usepackage[T1]{fontenc}
\usepackage{lmodern}

\usepackage{hyperref}

\usepackage{graphicx}
\usepackage{caption}
\usepackage{subfig}
\usepackage{float}

\usepackage{tikz}
\usepackage[percent]{overpic}
\usepackage{epstopdf} 

\usepackage[utf8]{inputenc}
\usepackage{amsmath, amsthm, amssymb}
\usepackage{bbm}
\usepackage{csquotes}
\usepackage[english]{babel}
\usepackage{marvosym}
\usepackage{color}
\usepackage{varwidth}
\usepackage{pgfplots}
\pgfplotsset{compat=1.13}

\usepackage[paperheight=235mm,paperwidth=200mm,textwidth=125mm,textheight=199mm]{geometry}

\newcommand{\mylabel}[2]{#2\def\@currentlabel{#2}\label{#1}}

\newtheorem{theorem}{Theorem}[section] 

\theoremstyle{definition} 
\newtheorem{definition}[theorem]{Definition} 

\newtheorem{lemma}[theorem]{Lemma} 

\theoremstyle{definition}  
\newtheorem{remark}[theorem]{Remark}

\newtheorem{corollary}[theorem]{Corollary} 

\theoremstyle{definition}

\theoremstyle{definition}

\theoremstyle{definition}

\newcommand{\be}{\begin{equation}}
\newcommand{\ee}{\end{equation}}
\newcommand{\bea}{\begin{eqnarray}}
\newcommand{\eea}{\end{eqnarray}}
\newcommand{\beann}{\begin{eqnarray*}}
	\newcommand{\eeann}{\end{eqnarray*}}
\newcommand{\benn}{\begin{equation*}}
\newcommand{\eenn}{\end{equation*}}

\newcommand{\cC}{{\mathcal C}}  

\def\txtd{{\textnormal{d}}}
\def\txte{{\textnormal{e}}}
\def\txti{{\textnormal{i}}}

\def\txtD{{\textnormal{D}}}

\usepackage{makeidx}
\makeindex

\begin{document}
	
\author{Christian Kuehn\thanks{CONTACT Christian Kuehn. 
			Email: ckuehn@ma.tum.de } and Christian M\"unch\thanks{CONTACT Christian M\"unch. Email: christian.muench@ma.tum.de}
		}

\title{Duck Traps: Two-dimensional Critical Manifolds in Planar Systems}	

\maketitle

\begin{abstract}
	In this work we consider two-dimensional critical manifolds in planar fast-slow systems near fold and so-called canard (=``duck'') points. These higher-dimension, and lower-codimension, situation is directly motivated by the case of hysteresis operators limiting onto fast-slow systems as well as by systems with constraints. We use geometric desingularization via blow-up to investigate two situations for the slow flow: generic fold (or jump) points, and canards in one-parameter families. We directly prove that the fold case is analogous to the classical fold involving a one-dimensional critical manifold. However, for the canard case, considerable differences and difficulties appear. Orbits can get trapped in the two-dimensional manifold after a canard-like passage thereby preventing small-amplitude oscillations generated by the singular Hopf bifurcation occurring in the classical canard case, as well as certain jump escapes.
\end{abstract}
	
	
\section{Introduction}
\label{Sec:Intro}

In this introductory section, we are going to explain the main concepts of fast-slow systems, 
non-hyperbolicity and the blow-up technique. Furthermore, we are going to motivate codimension zero 
critical manifolds as considered in this work arising from hysteresis operators. Then we present 
our main results on an informal level. A typical planar fast-slow system takes the form
\begin{alignat}{2}
\varepsilon \frac{\txtd x}{\txtd\tau} = \varepsilon \dot{x} &= 
f(x,y,\varepsilon),\label{Eq:fastslowIntro_x}\\
\frac{\txtd y}{\txtd\tau} = \dot{y} &= g(x,y,\varepsilon),
\label{Eq:fastslowIntro_y}
\end{alignat}
for $x\in\mathbb{R}^m,y\in\mathbb{R}^n$, with a small parameter $0<\varepsilon\ll 1$ and 
given initial conditions. In the classical case, $f,g$ are assumed to be sufficiently 
smooth functions~\cite{Fenichel4,Jones,KuehnBook}. Later on, we are going to restrict 
attention to the piecewise smooth case. Since $\varepsilon$ is small, one can consider 
first the singular limit $\varepsilon=0$ in \eqref{Eq:fastslowIntro_x}--\eqref{Eq:fastslowIntro_y}, 
which gives the reduced system (or slow subsystem)
\begin{alignat}{2}
0 &= f(x,y,0),\label{Eq:ReducedfastslowIntro_x}\\
\dot{y} &= g(x,y,0),\label{Eq:ReducedfastslowIntro_y}
\end{alignat}
which is a differential-algebraic equation. Here, the critical manifold 
$$\mathcal{C}_0:=\{(x,y)\in \mathbb{R}^{m+n}:\,f(x,y,0)=0\}$$\index{a@$\mathcal{C}_0$} 
generically has codimension $m$, see (c) in Figure~\ref{Fig:blow-up}.
Substituting the fast time scale $t:=\tau/\varepsilon$ 
into~\eqref{Eq:fastslowIntro_x}--\eqref{Eq:fastslowIntro_y} yields the system
\begin{alignat}{2}
x' &= 
f(x,y,\varepsilon),\label{Eq:fastsubIntro_x}\\
y' &= \varepsilon g(x,y,\varepsilon),\label{Eq:fastsubIntro_y}
\end{alignat}
and this time the limit $\varepsilon=0$ takes the form of a parametrized differential 
equation, the layer problem (or fast subsystem)
\begin{alignat}{2}
x' &= f(x,y,0),\label{Eq:layerIntro_x}\\
y' &= 0.\label{Eq:layerIntro_y}
\end{alignat}
Intuitively, we expect that a singular limit of solutions 
for \eqref{Eq:fastslowIntro_x}--\eqref{Eq:fastslowIntro_y}, 
or equivalently \eqref{Eq:fastsubIntro_x}--\eqref{Eq:fastsubIntro_y}, 
as $\varepsilon\rightarrow 0$ should qualitatively behave according to some 
weighted mixture between 
\eqref{Eq:ReducedfastslowIntro_x}--\eqref{Eq:ReducedfastslowIntro_y} 
and \eqref{Eq:layerIntro_x}--\eqref{Eq:layerIntro_y}.
For sufficiently smooth functions $f,g$, this is indeed the case.
In particular, for sufficiently small $\varepsilon>0$ and outside 
$\mathcal{O}(\varepsilon)$-distance of the critical manifold $\{f(x,y,0)=0\}$, 
the solution of \eqref{Eq:fastsubIntro_x}--\eqref{Eq:fastsubIntro_y} is well 
approximated by the layer problem, i.e., the $x$-components evolve fast while 
the $y$-components remain almost constant.
A subset $M\subset\cC_0$ is called normally hyperbolic,
if $\txtD_x f(p,0)$ is a hyperbolic matrix for all 
$p\in M$~\cite{Fenichel4,Jones,KuehnBook}.
In particular, an equilibrium $(x_*,y_*)$ 
of \eqref{Eq:fastsubIntro_x}--\eqref{Eq:fastsubIntro_y} in $\cC_0$ is 
normally hyperbolic if $\txtD_x f(x_*,y_*,0)$ has no eigenvalues with zero real 
parts. With this definition, in $\mathcal{O}(\varepsilon)$-distance of any normally 
hyperbolic compact subset of the critical manifold, Fenichel theory yields the 
existence of a perturbed invariant slow manifold, such that the 
system \eqref{Eq:fastsubIntro_x}--\eqref{Eq:fastsubIntro_y} restricted to this 
manifold behaves similar to the reduced flow \cite{Fenichel4}, see 
also \cite{Jones,KuehnBook,WigginsIM} for Fenichel theory and its 
broader context. Additionally, the slow manifold has the same stability properties 
as the critical manifold, i.e., it is attracting/repelling for solutions 
of \eqref{Eq:fastsubIntro_x}--\eqref{Eq:fastsubIntro_y} for different directions, in 
the same way as the original subset $\cC_0$ was for the layer 
problem \eqref{Eq:layerIntro_x}--\eqref{Eq:layerIntro_y}.\medskip

The situation gets more complicated if the critical manifold contains non-hyperbolic 
points. A very powerful tool to analyze the dynamics close to such points is the 
so-called blow-up technique used in several contexts in 
fast-slow systems~\cite{Dumortier1,DumortierRoussarie,KruSzm3,KuehnBook}.
The idea is to transform the (extended) fast-slow 
system \eqref{Eq:fastsubIntro_x}--\eqref{Eq:fastsubIntro_y} given by
\begin{alignat*}{2}
x' &= f(x,y,\varepsilon),\\
y' &= \varepsilon g(x,y,\varepsilon),\\
\varepsilon' &=0,
\end{alignat*}
in such a way, that a non-hyperbolic equilibrium point $(x_*,y_*,0)$ is blown-up, 
e.g., to a whole sphere. Frequently, the resulting desingularized system contains 
only (partially) hyperbolic equilibria, which improves the situation dynamically, i.e.,
we have gained hyperbolicity. For illustration, let us consider the case $m=1=n$ (as 
applied in this work, see also \cite{KruSzm3,BenoitCallotDienerDiener,DumortierRoussarie}) 
and the polar blow-up of the system 
in case of a canard point. This system depends on an additional parameter $\lambda$, which
we can also append via $\lambda'=0$, and one is interested in the local dynamics around 
the equilibrium $(0,0,0,0)$ (the canard point) for small $\varepsilon$ and upon variation 
of $\lambda$. In particular, the system is blown-up to $(\bar{x},\bar{y},\bar{\varepsilon},
\bar{\lambda},\bar{r})\in\mathbb{S}^2\times[-\mu,\mu]\times[0,\rho]$ by the transformation 
$\Psi$, which is defined by 
\begin{equation}
\label{Blow_up_Psi}
x=\bar{r}\bar{x},\quad y=\bar{r}^2\bar{y},\quad 
\varepsilon=\bar{r}^2\bar{\varepsilon},\quad \lambda=\bar{r}\bar{\lambda}.
\index{b@$\Psi$| {\eqref{Blow_up_Psi}, }}
\end{equation}
Here, for $\bar{r}=0$ and all $\bar{\lambda}\in [-\mu,\mu]$, $\Psi$ maps the sphere
$\mathbb{S}^2=\{\bar{x}^2+\bar{y}^2+\bar{\varepsilon}^2=1\}$ \index{c@$\mathbb{S}^2$} to 
the canard point $(0,0,0,0)$, see (a) in Figure~\ref{Fig:blow-up}.

\begin{figure}
	[htbp]
	\centering
	\subfloat[]
	{
		\begin{overpic}[width=0.3\textwidth]{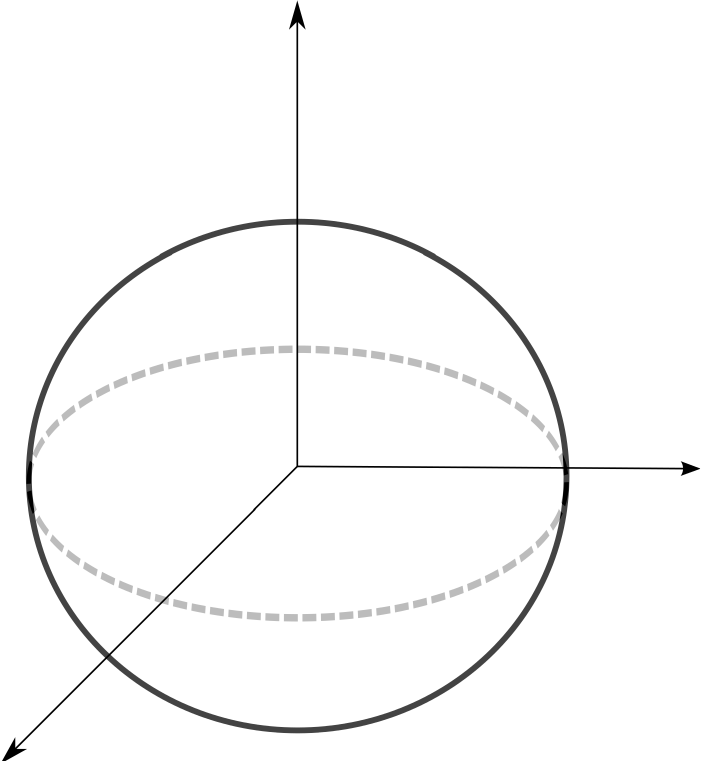}
			\put(42,95){\scalebox{1.0}{$\bar{\varepsilon}$}}
			\put(-5,3){\scalebox{1.0}{$\bar{x}$}}
			\put(88,30){\scalebox{1.0}{$\bar{y}$}}
			\put(68,10){\scalebox{1.0}{$\mathbb{S}^2$}}
		\end{overpic}
	}
	\subfloat[]
	{
		\begin{overpic}[width=0.3\textwidth]{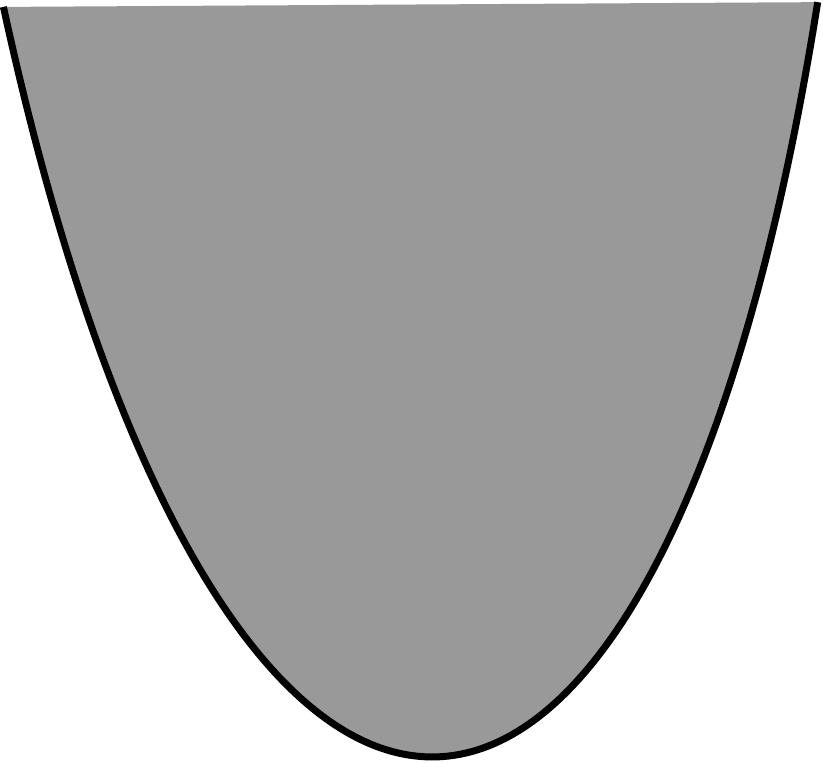}
			\put(50,40){\scalebox{1.0}{$\mathcal{C}_0$}}
		\end{overpic}
	}
	\subfloat[]
	{
		\begin{overpic}[width=0.3\textwidth]{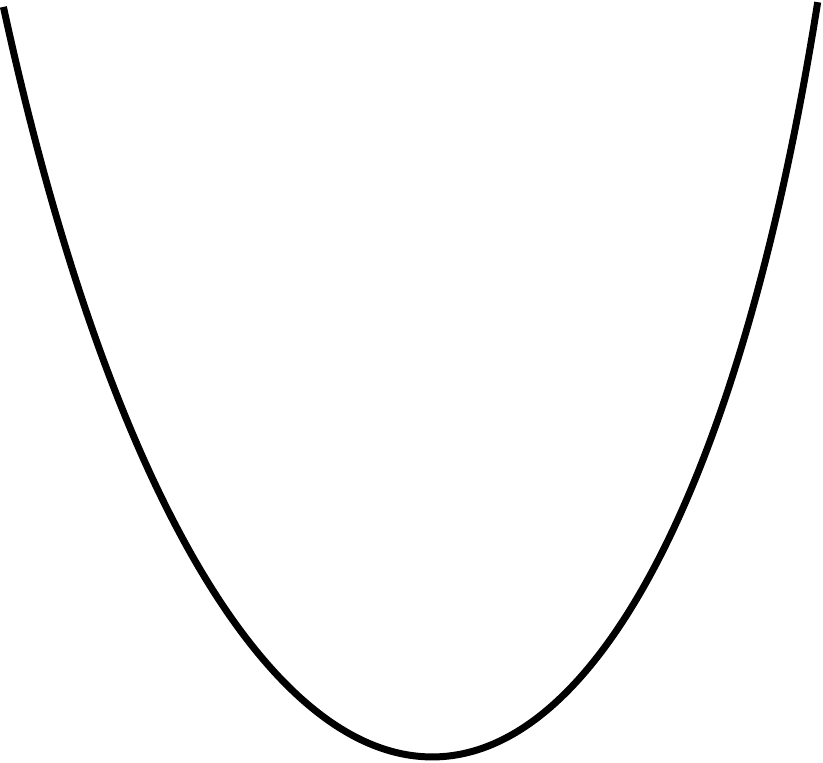}
			\put(10,25){\scalebox{1.0}{$\mathcal{C}_0$}}
		\end{overpic}
	}
	\caption{
		\label{Fig:blow-up}
		(a): Projection to $(\bar{x},\bar{y},\bar{\varepsilon})$-coordinates of the 
		reduced system $\bar{r}=0$ of the weighted polar blow-up $\Psi$. The canard 
		point $(0,0,0,0)$ is blown-up to $\mathbb{S}^2\times[-\mu,\mu]$.
		(b): Example of a critical manifold $\mathcal{C}_0$ (gray) with codimension 
		zero in the plane $m=1=n$. 
		(c): Example of a critical manifold $\mathcal{C}_0$ with codimension one in 
		the plane $m=1=n$.}
\end{figure}	

Typically, the dynamics for the blown-up system is analyzed in different (overlapping) 
directional charts. Fenichel theory can be applied in the regions, which correspond to 
parts far enough away from any non-hyperbolic point 
of \eqref{Eq:fastsubIntro_x}--\eqref{Eq:fastsubIntro_y}, and here the dynamics is 
analyzed in directional charts corresponding to $\bar{y},\bar{x}=\pm 1$, depending on 
the signs of $x,y$ in these areas. For the canard point, we need one chart $K_1$ for 
the direction $\overline{y}=1$. The corresponding transformation map $\Phi_1$ is 
determined by 
\begin{equation}
\label{Blow_up_Phi1}
x=r_1x_1,\quad y=r_1^2,\quad \varepsilon = r_1^2\varepsilon_1,\quad \lambda = r_1\lambda_1,
\index{d@$\Phi_1$| {\eqref{Blow_up_Phi1}, }}
\end{equation}
with domain
$$
V_1= (-x_{1,0},x_{1,0})\times [-\rho,\rho] \times [0,1) \times (-\mu,\mu).
\index{e@$V_1$}
$$
Here, $x_{1,0}>0$ is chosen sufficiently large, and $\rho >0,\mu >0$ are chosen sufficiently 
small. That is, $\varepsilon \in [0,\varepsilon_0)$ with $\varepsilon_0 = \rho^2$.
Chart $K_1$ is necessary to study \eqref{Eq:blowup_x}--\eqref{Eq:blowup_lamb} in sets of 
the form $\{(x,y):\, y\in (\varepsilon,\rho^2],\ x\in (-x_{1,0}\sqrt{y},x_{1,0}\sqrt{y})\}$.
For $\varepsilon\in (0,\rho^2]$, we define 
$$
V_{1,\varepsilon}:=\{(x_1,r_1,\varepsilon_1,\lambda_1)\in V_1:\, r_1^2\varepsilon_1=\varepsilon\}.
\index{f@$V_{1,\varepsilon}$}
$$
The dynamics close to the equilibrium point is usually analyzed with help of a rescaling (or 
classical) chart $K_2$, that is, with the directional chart corresponding to $\bar{\varepsilon}=1$.
In the canard case, the corresponding transformation map $\Phi_2$ is determined by
\begin{equation}
\label{Blow_up_Phi2}
x=r_2x_2,\quad y=r_2^2y_2,\quad \varepsilon = r_2^2,\quad \lambda = r_2\lambda_2,
\index{g@$\Phi_2$| {\eqref{Blow_up_Phi2}, }}
\end{equation}
with domain
$$
V_2= D\times [0,\rho] \times (-\mu,\mu).
\index{h@$V_2$}
$$
Here, $D$ is a sufficiently large disc with center $(0,0)$.
Chart $K_2$ is applied to study the system in a neighbourhood of the origin of 
size $\mathcal{O}(\sqrt{\varepsilon})$ in $x$-direction and $\mathcal{O}(\varepsilon)$ 
in $y$-direction. For $\varepsilon\in (0,\rho^2]$, we define 
$$
V_{2,\varepsilon}:=\{(x_2,y_2,r_2,\lambda_2)\in V_2:\, r_2^2=\varepsilon\}.
\index{i@$V_{2,\varepsilon}$}
$$
The analysis will be carried out in these sets in the blow-up, and it is 
helpful to already introduce two other sets (see also Figure~\ref{Fig:VandVeps}):

\begin{definition}\label{Def:nbhd}
	For $\varepsilon\in (0,\rho^2]$ and $P_{(x,y)}$ the projection to $(x,y)$-coordinates we 
	define the set
	$$
	V_{\varepsilon}:=P_{(x,y)}(\Phi_1(V_{1,\varepsilon}))\cup P_{(x,y)}(\Phi_2(V_{2,\varepsilon})).
	$$
	Moreover, we define $V:=V_{\varepsilon_0}=V_{\rho^2}$.
	\index{j@$V_\varepsilon$ and $V$| {Definition~\ref{Def:nbhd}, }}
\end{definition}

\begin{figure}
	[htbp]
	\centering
	\subfloat[]
	{
		\begin{overpic}[width=0.41\textwidth]{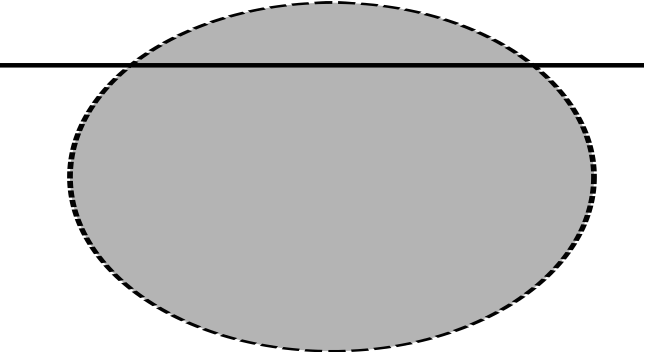}
			\put(85,48){\scalebox{0.7}{$P_{(x,y)}(\Phi_1(V_{1,\rho^2}))$}}
			\put(30,20){\scalebox{0.7}{$P_{(x,y)}(\Phi_2(V_{2,\rho^2}))$}}
		\end{overpic}
	}
	\qquad \qquad
	\subfloat[]
	{
		\begin{overpic}[width=0.41\textwidth]{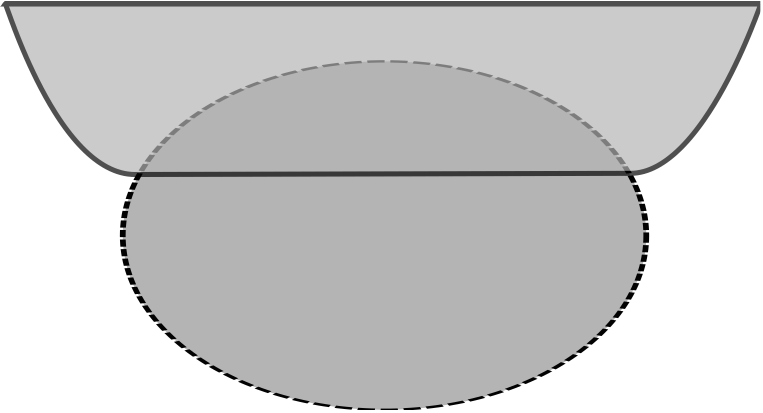}
			\put(50,49){\scalebox{0.7}{$P_{(x,y)}(\Phi_1(V_{1,\varepsilon}))$}}
			\put(35,20){\scalebox{0.7}{$P_{(x,y)}(\Phi_2(V_{2,\varepsilon}))$}}
		\end{overpic}
	}
	\caption{
		\label{Fig:VandVeps}
		(a): Sketch of $V=V_{\rho^2}$. In particular, $V$ is the union of the 
		line $P_{(x,y)}(\Phi_1(V_{1,\rho^2}))=\{(x,y):\, y=\rho^2,\ x\in (-x_{1,0}\sqrt{y},
		x_{1,0}\sqrt{y})\}$, and the ellipse $P_{(x,y)}(\Phi_2(V_{2,\rho^2}))$ (dashed gray) 
		which has width of order $\mathcal{O}(\rho)$ and height of order $\mathcal{O}(\rho^2)$.
		(b): Sketch of $V_{\varepsilon}$ for $\varepsilon\in (0, \rho^2)$. $V_{\varepsilon}$ 
		is the union of $P_{(x,y)}(\Phi_1(V_{1,\rho^2}))=\{(x,y):\, y\in(\varepsilon,\rho^2],
		\ x\in (-x_{1,0}\sqrt{y},x_{1,0}\sqrt{y})\}$ (solid light gray) and the ellipse 
		$P_{(x,y)}(\Phi_2(V_{2,\rho^2}))$ (dashed dark gray) which has width of order 
		$\mathcal{O}(\sqrt{\varepsilon})$ and height of order $\mathcal{O}(\varepsilon)$.
	}
\end{figure}

We also remark that for some problems, it 
requires more than one blow-up until all non-hyperbolic points (arising during 
the iterated blow-ups) have gained enough hyperbolicity. For the results in this work, one 
weighted polar blow-up as outlined above will be sufficient.\medskip

Different from the classical assumption, we analyze fast-slow systems with critical manifold 
of codimension zero for $m=1=n$, i.e. $\mathcal{C}_0$ has dimension two, see (b) in 
Figure~\ref{Fig:blow-up}. This implies that the non-linearity $f$ is only piecewise smooth, yet
our case is substantially different from other piecewise smooth fold 
cases (see e.g.~\cite{Desrochesetal1,KuehnScaleSN,RobertsGlendinning}), where the critical manifold
is less regular but still of dimension one. For our case of a dimension two critical manifold, 
we have to carefully analyze the (local) dynamics of 
system \eqref{Eq:fastsubIntro_x}--\eqref{Eq:fastsubIntro_y} on either side of the particular 
manifold where $f$ is not differentiable. The equilibrium point, which is non-hyperbolic for 
the smooth system, will be located in the separating manifold for the two smooth
regions. It turns out, that the same blow-up as for the smooth system can yield the desired 
results, but interesting additional and novel phenomena are observed.

In particular, the curvature of the function $g$, which determines the slow flow, is crucial 
now. More precisely, if $x\mapsto g(x,y,0)$, $(x,y)\in \mathcal{C}_0$ is not one-to-one, then 
the fast-slow system has several equilibria within $\mathcal{C}_0$. Hence, different from the 
classical case, higher order terms of $f$ and $g$ become relevant for the local analysis around 
the non-hyperbolic equilibrium. Fast-slow systems with critical manifold of codimension zero 
are relevant in several applications. Amongst others, the coupling of (systems) of ordinary 
differential equations (ODEs) or partial differential equations (PDEs) and a hysteresis operator 
are often approximated by the singular limit of fast-slow systems, see 
e.g.~\cite{KuehnMuench,MielkeRoubicek,Mielke3}. A particular subclass of interesting scalar 
hysteresis operators are so-called generalized play operators \cite{BrokateSprekels,Visintin}.
Given two increasing and (piecewise) smooth functions
\benn
L,U:\mathbb{R}\rightarrow \mathbb{R}, \qquad L<U, 
\eenn
the corresponding generalized play operator $\mathcal{P}$ maps time-continuous functions 
$y=y(t)$ to time-continuous functions $\mathcal{P}[y]$. Here, $\mathcal{P}[y]$ remains 
constant while $(y,\mathcal{P}[y])$ is located between the graphs $(y,U(y))$ and $(y,L(y))$. 
Once $\mathcal{P}[y]=L(y)$, so that the right curve $(y,L(y))$ is touched by $(y,\mathcal{P}[y])$ 
in phase space, then $\mathcal{P}[y]$ increases according to $\mathcal{P}[y]=L(y)$ as long 
as $y$ is monotone increasing. Similarly, $\mathcal{P}[y]=U(y)$ holds as long as $y$ is decreasing 
once $(y,\mathcal{P}[y])$ touches the left graph $(y,U(y))$. In other words, $\mathcal{P}[y]$ 
remains constant if located strictly between the graphs of $L$ and $U$, and otherwise moves on 
the graphs, see (a) in Figure~\ref{Fig:play}. 

\begin{figure}
	[htbp]
	\centering
	\subfloat[]
	{
		\begin{overpic}[width=0.4\textwidth]{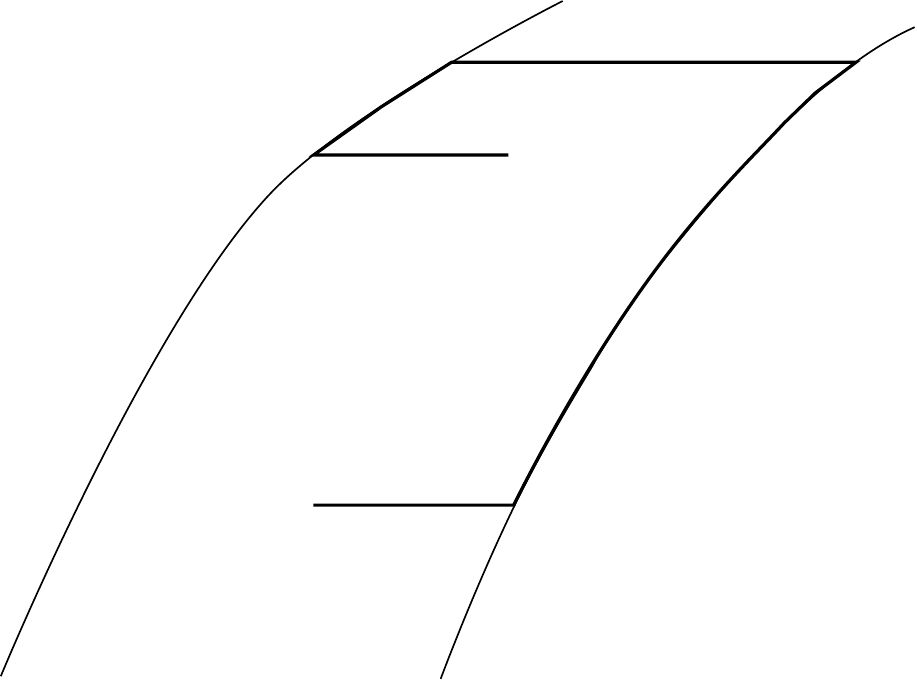}
			\put(4,2){\scalebox{1.0}{$U$}}
			\put(51,2){\scalebox{1.0}{$L$}}
			\put(35,19){
				\begin{tikzpicture}
				\draw[->] (0,0) -- (0.001,0);
				\end{tikzpicture}}
			\put(65,39.7){
				\begin{tikzpicture}
				\draw[->] (0,0) -- (0.001,0.002);
				\end{tikzpicture}}
			\put(65,67.4){
				\begin{tikzpicture}
				\draw[->] (0,0) -- (-0.001,0);
				\end{tikzpicture}}
			\put(38.5,62.5){
				\begin{tikzpicture}
				\draw[->] (0,0) -- (-0.001,-0.0005);
				\end{tikzpicture}}
			\put(42,57.1){
				\begin{tikzpicture}
				\draw[->] (0,0) -- (0.001,0);
				\end{tikzpicture}}
		\end{overpic}
	}
	\subfloat[]
	{
		\begin{overpic}[width=0.4\textwidth]{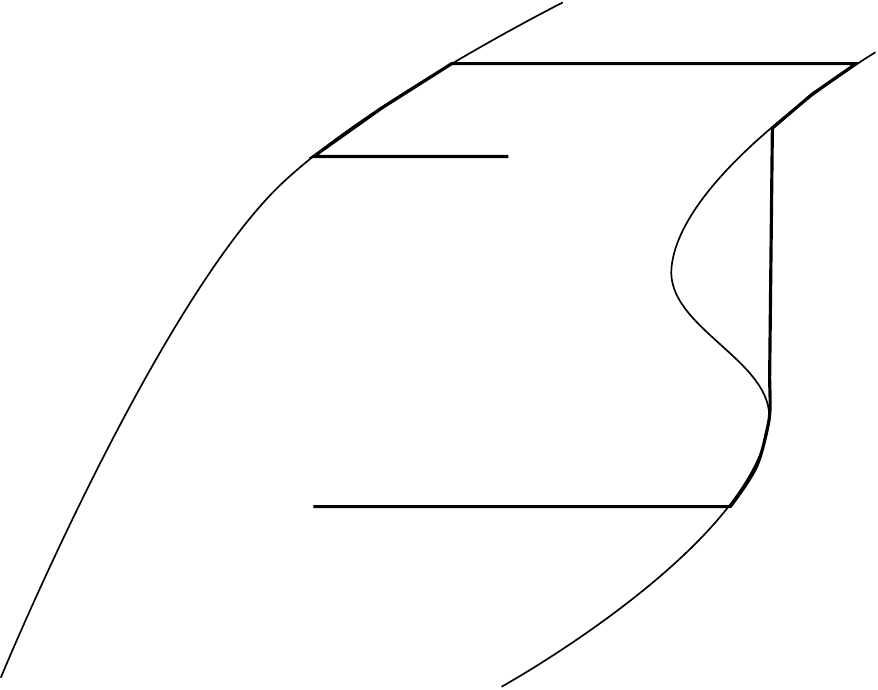}
			\put(4,2){\scalebox{1.0}{$U$}}
			\put(51,2){\scalebox{1.0}{$L$}}
			\put(35,20.58){
				\begin{tikzpicture}
				\draw[->] (0,0) -- (0.001,0);
				\end{tikzpicture}}
			\put(84,26.5){
				\begin{tikzpicture}
				\draw[->] (0,0) -- (0.001,0.003);
				\end{tikzpicture}}
			\put(85,42){
				\begin{tikzpicture}
				\draw[->>] (0,0) -- (0,0.2);
				\end{tikzpicture}}
			\put(89,67.1){
				\begin{tikzpicture}
				\draw[->] (0,0) -- (0.001,0.001);
				\end{tikzpicture}}
			\put(65,71.3){
				\begin{tikzpicture}
				\draw[->] (0,0) -- (-0.001,0);
				\end{tikzpicture}}
			\put(38.5,64.5){
				\begin{tikzpicture}
				\draw[->] (0,0) -- (-0.001,-0.001);
				\end{tikzpicture}}
			\put(42,60.5){
				\begin{tikzpicture}
				\draw[->] (0,0) -- (0.001,0);
				\end{tikzpicture}}
		\end{overpic}
	}
	\caption{
		\label{Fig:play}	
		(a): Typical trajectory $(y,\mathcal{P}[y])$ in phase space of a classical 
		generalized play operator, which is defined by two monotone increasing curves 
		$L$ and $U$.
		(b): Situation if one boundary curve (here $L$) contains a fold point. In 
		particular, $L$ is not monotone increasing. The solution behaves as a trajectory 
		of a classical generalized play operator until it reaches the fold point. Then it 
		jumps up, similar as trajectories of a relay operator, until it hits the graph of 
		$L$ again. Afterward, the solution continues as a trajectory of a classical 
		generalized play operator.
	}
\end{figure}

Finally, $L(y)\leq \mathcal{P}[y] \leq U(y)$ holds at all times. In particular, a 
prescribed fixed initial value $z_0\in \mathbb{R}$, which can be located outside 
of $[L(y(0)),U(y(0))]$, is instantaneously projected to $[L(y(0)),U(y(0))]$ according to 
$$
\mathcal{P}[y](0) = \min\{\max\{L(y(0)), z_0\}, U(y(0))\}.
$$
For the approximation by fast-slow systems, the non-linearity $f$ which determines 
the fast flow $x$ is therefore chosen such that its sign is negative for input 
parameters above the graph $(y,U(y))$ and positive below $(y,L(y))$, so that the 
fast flow is directed towards the area between both graphs. 
Note that $x(0)=z_0\in \mathbb{R}$ can be located outside of $[L(y(0)),U(y(0))]$. 
Choosing $f=0$ between $(y,U(y))$ and $(y,L(y))$, it turns out that the fast variable 
$x$ tends to $\mathcal{P}[y]$ in the singular limit $\varepsilon\rightarrow 0$, 
while $y$ still solves the slow equation. In particular, the two-dimensional area 
between the graphs of $L$ and $U$ is the critical manifold of the approximating 
fast-slow system, i.e., it has codimension zero. Note that in order to obtain interesting 
long-term effects in this setting, one has to allow time-dependent $g$, since 
$\mathcal{P}[y]$ is unchanged while $y$ remains constant. However, before considering 
non-autonomous systems, it is crucial to understand the dynamics of the autonomous 
system \eqref{Eq:fastsubIntro_x}--\eqref{Eq:fastsubIntro_y} first. If one drops the 
assumption that $U$ and $L$ have to be monotone increasing, then a non-hyperbolic 
point such as a fold can be generically located on e.g. $L$ (or $U$).
In this case, with $\varepsilon$ small, some trajectories of the approximating 
fast-slow system, follow $(y,L(y))$ in $(y,x)$-phase space according to the slow 
dynamics up to the fold point and then jump upwards according to the fast dynamics 
until eventually arriving at a point on the graph $(y,L(y))$ again, see (b) in 
Figure~\ref{Fig:play}.

In this work, we restrict to the local dynamics around one boundary curve of the 
critical manifold (i.e., we select $U$ or $L$) which acts attracting from one side. 
In particular, we analyze the local dynamics of the corresponding system if this 
curve contains a non-hyperbolic point, as in (b) in Figure~\ref{Fig:play}.\medskip

The outline of this work is the following.
In Section~\ref{Sec:fold}, we analyze the analogous behaviour to a so-called 
fold singularity in the codimension zero setting. In particular, we extend the critical 
manifold 
$$
\mathcal{C}_{\partial}:=\{(x,y)\in \mathbb{R}^2:\, y=x^2\}
\index{k@$\mathcal{C}_{\partial}$}
$$
to one side by replacing $f(x,y)=-y+x^2$ by zero on one side of $\mathcal{C}_{\partial}$. 
This yields a fast-slow system with critical manifold 
$$
\mathcal{C}_0=\{(x,y)\in \mathbb{R}^2:\, y\geq x^2\},
$$ which has the classical form of a system with a fold singularity, but only from 
one side. In Theorem~\ref{Thm:class_fold} we restate a classical result for the local 
dynamics of the slow flow close to the fold point, see \cite{DumortierRoussarie,KruSzm3,
	KruSzm2,KuehnBook}. In Theorem~\ref{Thm:fold}, we prove that this classical result 
still applies in the codimension zero setting, by proving that the set 
$\mathbb{R}^2\backslash\mathcal{C}_0$ is invariant. In particular, Theorem~\ref{Thm:fold} 
also applies if $f(x,y)$ is not replaced by the zero function in $\mathcal{C}_0$, but by 
any other function $h(x,y)$. The result for $h=0$ is summarized in Corollary~\ref{Cor:canard}.

In Sections~\ref{Sec:canard_class}--\ref{Sec:canard}, we analyze the analogous case to a 
canard singularity in the codimension zero setting, which turns out to be far more difficult.
Again, $f(x,y)=-y+x^2$ is set to zero on one side of $\mathcal{C}_{\partial}$, and the 
resulting fast-slow system has again critical manifold 
$\mathcal{C}_0=\{(x,y)\in \mathbb{R}^2:\, y\geq x^2\}$.
The function $g$ this time takes the form 
$$
g=xg_1- \lambda g_2 + yg_3,\qquad \text{with}\qquad g_1,g_2 = 1+\mathcal{O}(x,y,\lambda),
$$
and we want to analyze the fast-slow system for $\varepsilon>0$ and for different 
parameters $\lambda$. In Section~\ref{Sec:canard_class}, we restate the results for 
the classical planar canard case \cite{KruSzm3,KruSzm2}.
Section~\ref{Sec:canard} contains our main results on the dynamics within the small 
neighbourhood $V$ of the equilibrium point $(0,0)\in\mathcal{C}_{\partial}$ for $\lambda=0$.
We write $p_e=p_e(\lambda)$ for the (perturbed) equilibrium. In Theorem~\ref{Thm:canard1}, 
we prove the existence of two critical values $\lambda_H,\lambda_{c}$ and a branch of 
equilibria in $\mathcal{C}_0$ emanating from $p_e$, such that all trajectories starting 
on the left side of a small neighbourhood $U^{0}=U^{0}(\lambda,\varepsilon)$ around this 
branch leave the critical manifold, re-enter it to the right of $U^{0}$ but close to $p_e$, 
and leave $V$ in $\mathcal{C}_0$, provided that $-\lambda_0<\lambda<\lambda_{H}$.
For $-\lambda_0 <\lambda <\lambda_*$, $\lambda_* \in (\lambda_{H},\lambda_c)$, this 
neighbourhood $U^{0}$ around the graph $(x,u_e(x))$ of roots of $g(\cdot,\cdot,\lambda)$ 
has width $\mathcal{O}(\varepsilon+\sqrt{\varepsilon}\lambda)$ in $V_{2,\varepsilon}$ and 
respectively $\mathcal{O}((u_e(x))^2 + \lambda u_e(x))$ for $(x,u_e(x))\in V_{1,\varepsilon}$. 
For $\lambda_*<\lambda <\lambda_0$, $U^{0}$ it has width $\mathcal{O}(\varepsilon^{3/2}+
{\sqrt{\varepsilon}}\lambda)$ in $V_{2,\varepsilon}$ and respectively $\mathcal{O}((u_e(x))^2 + 
\lambda u_e(x))$ for $(x,u_e(x))\in V_{1,\varepsilon}$. The result follows since $p_e$ is 
attracting in the classical case for all $-\lambda_0<\lambda<\lambda_{H}$. Moreover, we prove 
that a maximal canard solution exists just as in the classical case if and only if 
$\lambda=\lambda_{c}$.\medskip 

In Theorem~\ref{Thm:canard2}, we show that stable half orbits below the parabola 
$\mathcal{C}_{\partial}$ occur as $\lambda$ passes through $\lambda_{H}$.
In particular, trajectories starting left of $U^{0}$ in $\mathcal{C}_0$ which exit 
the critical manifold outside of such a half orbit, are attracted to the half orbit 
from the outside and re-enter $\mathcal{C}_0$ to the right of it before they leave 
$V$ in $\mathcal{C}_0$. Similarly, trajectories which reach the parabola inside the 
half orbit stay there as long as they are located outside of the critical manifold, 
and re-enter $\mathcal{C}_0$ inside the half orbit before they leave $V$ in $\mathcal{C}_0$.
As the periodic orbits in the classical setting, also the half orbits get larger 
as $\lambda>\lambda_{H}$ increases up to a critical value $\lambda_{sc}<\lambda_{c}$, 
which is close to the maximal canard value, see Figure~\ref{Fig:numerics} for a numerical 
example. Yet, note carefully that these (non-maximal) canards essentially get trapped
in $\cC_0$ showing a very significant difference to the classical case.

Finally, in Theorem~\ref{Thm:canard3}, we analyze the behaviour for large values 
$\lambda_{c}<\lambda<\lambda_0$. In particular, we prove the existence of a vertical 
line $P_c$ in $(x,y)$-space, such that all solutions starting to the left of $P_c$ 
leave the neighbourhood $V$ below the critical manifold. All solutions starting 
in $\mathcal{C}_0\backslash U^{0}$ and to the right of $P_c$ leave $V$ in $\mathcal{C}_0$ 
and right of the branch of $U^{0}$, see Figure~\ref{Fig:LargeLambda}. Hence, there 
is again a trapping effect for part of the orbits.\medskip

In summary, our results provide a complete characterization of the local dynamics around 
fold and canard points in planar fast-slow systems $m=1=n$, when the critical manifold 
$\mathcal{C}_0$ has codimension zero. Furthermore, we develop a refinement of the blow-up
near folds applicable to higher-codimension situations, which may also appear in a wide
variety of other contexts, not only for hysteresis operators.

\section{Blow-up technique for the fold}\label{Sec:fold}

\subsection{The classical case}

The classical fast-slow system for the fold normal form is given by 
\index{l@classical fold| {\eqref{Eq:class_fold_x}--\eqref{Eq:class_fold_y}, }}
\begin{alignat}{2}
x' &= 
-y + x^2,\label{Eq:class_fold_x}\\
y' &= -\varepsilon.\label{Eq:class_fold_y}
\end{alignat}
The critical manifold for this system is given by the parabola $$\mathcal{C}_{\partial}:=\{(x,y)\in \mathbb{R}^2: y= x^2\},$$ 
i.e.~by a one-dimensional manifold in a two-dimensional dynamical system.
We also introduce the attracting and repelling branch $$\mathcal{C}_{0}^a:=\mathcal{C}_{\partial}\cap (\mathbb{R}^-\times\mathbb{R}),\qquad
\mathcal{C}_{0}^r:=\mathcal{C}_{\partial}\cap (\mathbb{R}^+\times\mathbb{R}).
\index{m@$\mathcal{C}_{0}^a$ and $\mathcal{C}_{0}^r$}$$
System \eqref{Eq:class_fold_x}--\eqref{Eq:class_fold_y}, and specifically its dynamics around the non-hyperbolic fold point $(0,0)$, has been analyzed in \cite{KruSzm3}.
In particular, a blow-up of the fold was applied to study the local behaviour around this point.
The main result \cite[Theorem~2.1]{KruSzm3} is the following:
Let $\rho>0$ be chosen small enough and consider sections
$$
\Delta_{\mathrm{in}}:=\{(x,\rho^2): x\in J\},\qquad \Delta_{\mathrm{out}}:=\{(\rho,y): y\in \mathbb{R}\}
$$
for some suitable interval $J\subset \mathbb{R}$.
Let $\Pi:\Delta_{\mathrm{in}}\rightarrow \Delta_{\mathrm{out}}$ be the transition map for the flow \eqref{Eq:class_fold_x}--\eqref{Eq:class_fold_y}.

\begin{theorem}\label{Thm:class_fold}
	There exists $\varepsilon_0 > 0$ such that the following assertions hold for $\varepsilon\in (0,\varepsilon_0)$:
	
	\begin{enumerate}
		\item[(T1)] The attracting slow manifold $\mathcal{C}_{\varepsilon}^a$ passes through $\Delta_{\mathrm{out}}$ at a point
		$(\rho,h(\varepsilon))$, where $h(\varepsilon)=\mathcal{O}\left(\varepsilon^{\frac{2}{3}}\right)$.
		\item[(T2)]
		The transition map $\Pi$ is a contraction with contraction rate
		$\mathcal{O}\left(e^{-\frac{c}{\varepsilon}}\right)$, where $c$ is a positive constant.
	\end{enumerate}
\end{theorem}

This result is shown by an appropriate blow-up transformation near the fold point.
Three directional charts $K_1$ to $K_3$ are necessary to describe the behaviour of trajectories close to the origin.
$K_1$ is determined by
$$
x=r_1x_1,\quad y=r_1^2,\quad \varepsilon=r_1^3\varepsilon_1.
$$
In the chart $K_1$, the desingularized vector field (i.e. divided by $r_1$) reads
\begin{alignat*}{2}
x_1' &= 
-1 + x_1^2 + \frac{1}{2}\varepsilon_1 x_1,\\
r_1' &= -\frac{1}{2}r_1\varepsilon_1,\\
\varepsilon_1' &= \frac{3}{2}\varepsilon_1^2.
\end{alignat*}
Similarly, the rescaling chart $K_2$ is determined by
$$
x=r_2x_2,\quad y=r_2^2y_2,\quad \varepsilon=r_2^3,
$$
and the desingularized vector field (i.e. divided by $r_2$) reads
\begin{alignat*}{2}
x_2' &=
- y_2 + x_2^2,\\
y_2' &= -1,\\
\dot{r}_2 &= 0.
\end{alignat*}
Finally, $K_3$ is determined by
$$
x=r_3,\quad y=r_3^2y_3,\quad \varepsilon=r_3^3\varepsilon_3,
$$ and the desingularized vector field (i.e. divided by $r_3$) reads
\begin{alignat*}{2}
r_3' &= 
r_3 (-y_3 + 1),\\
y_3' &= 
-\varepsilon_3 - 2y_3(-y_3+1),\\
\varepsilon_3' &= 
-3\varepsilon_3 (1-y_3).
\end{alignat*}

\subsection{Two-dimensional critical manifold}

The situation is somewhat different to the classic fold if the critical manifold has codimension zero.
Consider the fast-slow system
\index{n@new fold| {\eqref{Eq:fold_x}--\eqref{Eq:fold_y}, }}
\begin{alignat}{2}
x' &= \left\{\begin{matrix}
-y + x^2&& \text{for } y< x^2,\\
h(x,y)&& \text{for } y\geq x^2,
\end{matrix}\right.\label{Eq:fold_x}\\
y' &= -\varepsilon.\label{Eq:fold_y}
\end{alignat}
For $h(x,y)=0$, the critical manifold is given by $\mathcal{C}_0=\{(x,y)\in \mathbb{R}^2: y\geq x^2\}$.
Moreover, for $h(x,y)=-y+x^2$, the critical manifold is given by $\mathcal{C}_{\partial}=\{(x,y)\in \mathbb{R}^2: y= x^2\}$ as for the classical fold. In this case, system \eqref{Eq:fold_x}--\eqref{Eq:fold_y} has a non-hyperbolic equilibrium at $(0,0)$, see Figure~\ref{fig:01}, and Theorem~\ref{Thm:class_fold} applies.

\begin{figure}
	[htbp]
	\centering
	\begin{overpic}[scale=0.3]{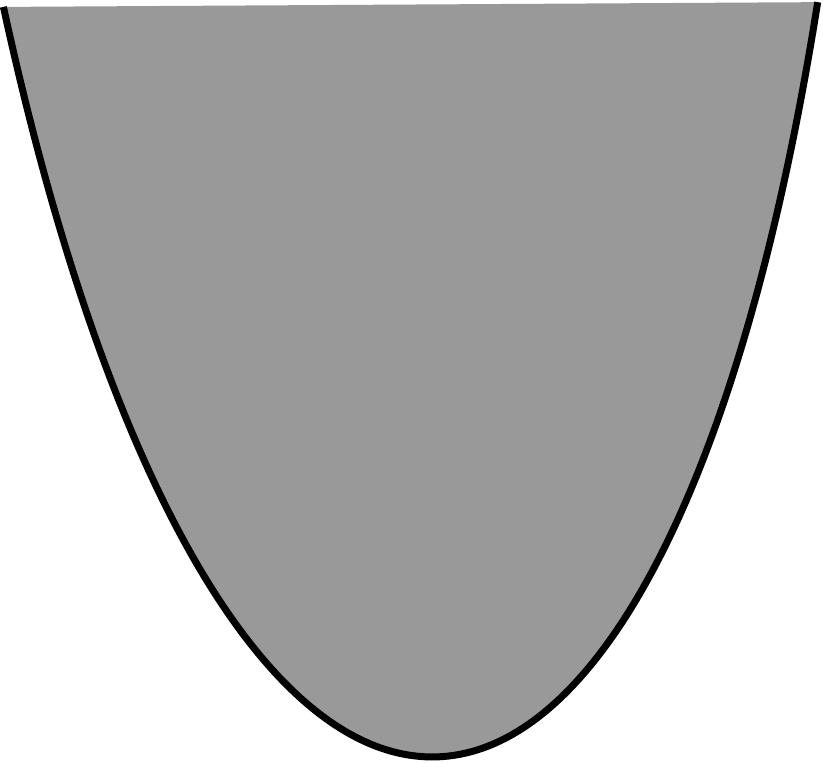}
		\put(50,40){\scalebox{1.0}{$\mathcal{C}_0$}}
		\put(10,25){\scalebox{1.0}{$\mathcal{C}_{\partial}$}}
		\put(36,-15){
			\begin{tikzpicture}
			\filldraw (0,0) circle (1pt) node[align=left,   below]{\scalebox{1.0}{$(0,0)$}};
			\end{tikzpicture}
		}
		\put(-40,-5){
			\begin{tikzpicture}
			\draw[->] (-1,0) -- (0,0) node[anchor = north]{$x$};
			\draw[->] (-1,0) -- (-1,1) node[anchor = west]{$y$};
			\end{tikzpicture}
		}
	\end{overpic}
	
	\caption{\\
		\label{fig:01}The two-dimensional set $\mathcal{C}_0$ (gray) depicts the critical manifold for the case $h(x,y)=0$. Its lower boundary $\mathcal{C}_{\partial}$ represents the critical manifold for the classical case $h(x,y)=-y+x^2$.
		The fold point, which will later be a canard point in Sections~\ref{Sec:canard_class}--\ref{Sec:canard}, is located at $(0,0)$.}
\end{figure}

We keep the notation $\mathcal{C}_0$ also for $h(x,y)\neq 0$.
It turns out that Theorem~\ref{Thm:class_fold} applies at least partly to the problem 
\eqref{Eq:fold_x}--\eqref{Eq:fold_y}.

As in the previous section, let $\rho>0$ be chosen small enough and consider sections
$$
\tilde{\Delta}_{\mathrm{in}}:=\{(x,\rho^2): x\in J\},\qquad \tilde{\Delta}_{\mathrm{out}}:=\{(\rho,y): y\in (-\infty,\rho^2)\}
$$
for some suitable interval $J\subset (-\infty,-\rho)$.
Let $\tilde{\Pi}:\tilde{\Delta}_{\mathrm{in}}\rightarrow \tilde{\Delta}_{\mathrm{out}}$ be the transition map for the flow \eqref{Eq:fold_x}--\eqref{Eq:fold_y}.

\begin{theorem}\label{Thm:fold}
	There exists $\varepsilon_0 > 0$ such that the following assertions hold for $\varepsilon\in (0,\varepsilon_0)$:
	
	$\tilde{\Pi}:\tilde{\Delta}_{\mathrm{in}}\rightarrow \tilde{\Delta}_{\mathrm{out}}$ is the restriction of the transition map $\Pi$ from the previous section to the set $\tilde{\Delta}_{\mathrm{in}}$.
	\begin{enumerate}
		\item[(T1)] The attracting slow manifold $\mathcal{C}_{\varepsilon}^a$ passes through $\tilde{\Delta}_{\mathrm{out}}$ at a point
		$(\rho,h(\varepsilon))$, where $h(\varepsilon)=O\left(\varepsilon^{\frac{2}{3}}\right)$.
		\item[(T2)]	The transition map $\Pi$ is a contraction with contraction rate
		$O\left(e^{-\frac{c}{\varepsilon}}\right)$, where $c$ is a positive constant.
	\end{enumerate}
\end{theorem}
In order to prove Theorem~\ref{Thm:fold} it is enough to show the following lemma:

\begin{lemma}\label{Lem:fold}
	System \eqref{Eq:fold_x}--\eqref{Eq:fold_y}, transformed into the different charts $K_1$--$K_3$, leaves the set $\mathbb{R}^2\backslash \mathcal{C}_0$ invariant. In particular, solution trajectories which start in $\mathbb{R}^2\backslash \mathcal{C}_0$ never enter $\mathcal{C}_0$.
\end{lemma}

\begin{proof}
	We show that the dynamics of trajectories starting in the left $(x,y)$-half plane, and close to but outside of $\mathcal{C}_0$, is independent of the particular vector field $h(x,y)$.
	In particular, for $\varepsilon>0$, each trajectory $(x_\varepsilon,y_\varepsilon)$ starting in $\mathbb{R}^2\backslash \mathcal{C}_0$ remains there.
	Since the blow-up technique in this case delivers the same results as for problem \eqref{Eq:fold_x}--\eqref{Eq:fold_y} with $h(x,y)=-y+x^2$ which has been analyzed in \cite{KruSzm3}, we only have to 
	transform \eqref{Eq:fold_x}--\eqref{Eq:fold_y} with general $h$ into the single charts and make sure that the critical set $\mathcal{C}_0$ remains untouched.
	In chart $K_1$, the desingularized vector field (i.e. divided by $r_1$) reads
	\begin{alignat}{2}
	x_1' &= \left\{\begin{matrix}
	-1 + x_1^2 + \frac{1}{2}\varepsilon_1 x_1,&& \text{for } 1< |x_1|,\\
	\frac{1}{r_1}h(r_1x_1,r_1^2) + \frac{1}{2}\varepsilon_1 x_1,&& \text{for } 1\geq |x_1|,
	\end{matrix}\right.\label{Eq:fold_x_1}\\
	r_1' &= -\frac{1}{2}r_1\varepsilon_1,\label{Eq:fold_r_1}\\
	\varepsilon_1' &= \frac{3}{2}\varepsilon_1^2.\label{Eq:fold_eps_1}
	\end{alignat}
	In chart $K_2$, the desingularized vector field (i.e. divided by $r_2$) reads
	\begin{alignat}{2}
	x_2' &= \left\{\begin{matrix}
	- y_2 + x_2^2,&& \text{for } y_2< x_2^2,\\
	\frac{1}{r_2}h(r_2x_2,r_2^2y_2),&& \text{for } y_2\geq x_2^2,
	\end{matrix}\right.\label{Eq:fold_x_2}\\
	y_2' &= -1,\label{Eq:fold_y_2}\\
	\dot{r}_2 &= 0.\label{Eq:fold_r_2}
	\end{alignat}
	Finally, in chart $K_3$, the desingularized vector field (i.e. divided by $r_3$) reads
	\begin{alignat}{2}
	r_3' &= \left\{\begin{matrix}
	r_3 (-y_3 + 1),&& \text{for } y_3< 1,\\
	\frac{1}{r_3}h(r_3,r_3^2y_3),&& \text{for } y_3\geq 1,
	\end{matrix}\right.\label{Eq:fold_r_3}\\
	y_3' &= \left\{\begin{matrix}
	-\varepsilon_3 - 2y_3(-y_3+1),&& \text{for } y_3< 1,\\
	-\varepsilon_3 - 2y_3\frac{1}{r_3^2}h(r_3,r_3^2y_3),&& \text{for } y_3\geq 1,
	\end{matrix}\right.\label{Eq:fold_y_3}\\
	\varepsilon_3' &= \left\{\begin{matrix}
	-3\varepsilon_3 (1-y_3),&& \text{for } y_3< 1,\\
	-3\varepsilon_3\frac{1}{r_3^2}h(r_3,r_3^2y_3),&& \text{for } y_3\geq 1.
	\end{matrix}\right.\label{Eq:fold_eps_3}
	\end{alignat}
	Note that the functions $\frac{1}{r_1}h(r_1x_1,r_1^2)$, $\frac{1}{r_2}h(r_2x_2,r_2^2y_2)$ and $\frac{1}{r_3^2}h(r_3,r_3^2y_3)$ are in general not bounded for $r_3\rightarrow 0$.
	Nevertheless, we are only interested in the regions $\{1<|x_1|\}$, $\{y_2<x_2^2\}$ and $\{y_3<1\}$. In particular, we prove that no trajectory starting outside the blown-up sets of $\mathcal{C}_0$, i.e. outside the sets $\{1\geq|x_1|\}$, $\{y_2\geq x_2^2\}$ and $\{y_3\geq 1\}$, reaches them.
	Hence, we do not have to consider the particular dynamics in these latter cases.
	For $h(x,y)=0$, there is no problem at all.
	Since $\varepsilon_1'>0$, the set $\{1<|x_1|\}$ is invariant under \eqref{Eq:fold_x_1}--\eqref{Eq:fold_eps_1}.
	Transformed into the chart $K_2$, the set $\{1<|x_1|\}$ is given by $\{y_2<x_2^2\}$.
	By \cite[Proposition~2.3~2,5 and Proposition~2.6~5]{KruSzm3} each trajectory starting in the set $\{y_2<x_2^2\}$ remains in this set.
	By \cite[Proposition~2.3~1 and 5]{KruSzm3}, the $y_2$-component of each such trajectory is asymptotic to some $y_r<0$ and $x_2$ gets positive at some time.
	The transformation $y_3=y_2x_2^{-2}<0$ \cite[Lemma~2.2]{KruSzm3} implies that we arrive in the set $\{y_3<1\}$.
	Finally, \cite[Proposition~2.11]{KruSzm3} proves that we remain in this set.
	Hence, the main result \cite[Theorem~2.1]{KruSzm3} applies for all solutions $(x_\varepsilon,y_\varepsilon)$ of \eqref{Eq:fold_x}--\eqref{Eq:fold_y} which start in $(\mathbb{R}^-\times \mathbb{R})\backslash \mathcal{C}_0$.
\end{proof}

Note that Lemma~\ref{Lem:fold} determines the whole flow near the fold point for the particular case $h(x,y) = 0$.
Indeed, for $h(x,y)=0$, each trajectory starting at a point $(x_\varepsilon(0),y_\varepsilon(0)) \in \mathcal{C}_0\cap (\mathbb{R}^-\times \mathbb{R})$ enters the set $(\mathbb{R}^-\times \mathbb{R})\backslash \mathcal{C}_0$ at $(x_\varepsilon(0),x_\varepsilon(0)^2)$ and then remains in this set.
Also for starting points $(x_\varepsilon(0),y_\varepsilon(0)) \in \mathcal{C}_0\cap (\mathbb{R}^+\times \mathbb{R})$, the trajectory $(x_\varepsilon,y_\varepsilon)$ leaves the set $\mathcal{C}_0$ at $(x_\varepsilon(0),x_\varepsilon(0)^2)$ and, since it does not pass close to the fold point, classical theory can be applied for the further analysis.
In particular, from $(x_\varepsilon(0),x_\varepsilon(0)^2)$, in first order approximation, the trajectory continues horizontally to the right on the line $\{(x,x_\varepsilon(0)^2):\ x\geq x_\varepsilon(0)\}$.
We summarize this special case in the following corollary:

\begin{corollary}\label{Cor:canard}
	Let the assumptions of Theorem~\ref{Thm:fold} hold and fix $\varepsilon\in (0,\varepsilon_0]$.
	Consider system \eqref{Eq:fold_x}--\eqref{Eq:fold_y} with $h(x,y)=0$.
	
	Each solution with $(x_\varepsilon(0),y_\varepsilon(0)) \in \mathcal{C}_0$ moves vertically downwards in $(x,y)$-phase space and enters the invariant set $\mathbb{R}^2\backslash \mathcal{C}_0$ at $(x_\varepsilon(0),x_\varepsilon(0)^2)$. 
	
	For $x_\varepsilon(0)<0$, the trajectory then stays $\mathcal{O}(\varepsilon)$-close to the set $\mathcal{C}_{\partial}$ until it reaches a point in $\Delta_{\mathrm{in}}$, from where Theorem~\ref{Thm:fold} applies.
	
	For $x_\varepsilon(0)>0$, from $(x_\varepsilon(0),x_\varepsilon(0)^2)$, the trajectory is approximated to leading order by the fast subsystem and continues toward the right in $(x,y)$-phase space.
\end{corollary}

\section{The classical canard case}\label{Sec:canard_class}

The normal form of a canard point is given by \cite{KuehnBook,KruSzm3,KruSzm2}
\index{o@classical canard| {\eqref{Eq:class_blowup_x}--\eqref{Eq:class_blowup_lamb}, }}
\begin{alignat}{2}
x' &=
-y + x^2,\label{Eq:class_blowup_x}\\
y' &= \varepsilon g(x,y,\lambda),\label{Eq:class_blowup_y}\\
\varepsilon' &= 0 ,\label{Eq:class_blowup_eps}\\
\lambda' &= 0,\label{Eq:class_blowup_lamb}
\end{alignat}
where $g=xg_1- \lambda g_2 + yg_3$ is $\mathrm{C}^r$-smooth for $r\geq 3$, $g_i=g_i(x,y,\lambda)$ for $i=1,2,3$, and with $g_1,g_2 = 1+\mathcal{O}(x,y,\lambda)$.
We denote $a_1=\partial_x g_1(0,0,0), a_2=g_3(0,0,0)$ and assume $a_1,a_2>0$.
Moreover, $\varepsilon\in (0,\varepsilon_0]$ and $\lambda\in [-\lambda_0,\lambda_0]$ for certain values $\varepsilon_0,\lambda_0 >0$.
As for the classical fold, all points of the one-dimensional critical manifold $\mathcal{C}_{\partial}=\{(x,y):\, y= x^2\}$ are normally hyperbolic, except for the canard point at the origin. While classical Fenichel theory can be applied near the hyperbolic subset of the critical manifold in order to determine the behaviour of the slow flow of \eqref{Eq:class_blowup_x}--\eqref{Eq:class_blowup_lamb}, the blow-up \eqref{Blow_up_Psi} of the origin with directional charts \eqref{Blow_up_Phi1}--\eqref{Blow_up_Phi2} yields the corresponding results for the flow close to the canard point. The main results are the following:
\begin{theorem}\label{Thm:class_canard1}\cite[Theorems~3.1-3.2]{KruSzm2}
	Suppose $\varepsilon_0$, $\lambda_0$ are sufficiently small and consider $V=V_{\varepsilon_0}$ as defined in Definition~\ref{Def:nbhd}.
	Then for all fixed $\varepsilon\in (0,\varepsilon_0]$, system \eqref{Eq:class_blowup_x}--\eqref{Eq:class_blowup_lamb} has exactly one equilibrium $p_e\in V$\index{p@$p_e$| {Theorem~\ref{Thm:class_canard1}, }}, which converges to the canard point as $(\varepsilon,\lambda)\rightarrow 0$.
	Furthermore, there exists a curve $\lambda_H(\sqrt{\varepsilon}) = -\frac{a_2}{2}\varepsilon + \mathcal{O}(\varepsilon^{\frac{3}{2}})$\index{q@$\lambda_H(\sqrt{\varepsilon})$| {Theorem~\ref{Thm:class_canard1}, }} such that $p_e$ is stable for $\lambda<\lambda_H(\sqrt{\varepsilon})$ and loses stability through a Hopf bifurcation as $\lambda$ passes through $\lambda_H(\sqrt{\varepsilon})$.
	Moreover, there exists a function smooth in $\sqrt{\varepsilon}$ given by
	$$\lambda_c(\sqrt{\varepsilon})=-\left(\frac{a_2}{2} + \frac{-2a_1 -2a_2}{8}\right)\varepsilon + \mathcal{O}(\varepsilon^{3/2})=\frac{a_1 - a_2}{4}\varepsilon + \mathcal{O}(\varepsilon^{3/2}),$$\index{r@$\lambda_c(\sqrt{\varepsilon})$| {Theorem~\ref{Thm:class_canard1}, }}
	such that the attracting slow manifold $\mathcal{C}_{\varepsilon}^a$ connects to the repelling slow manifold $\mathcal{C}_{\varepsilon}^r$ if and only if $\lambda=\lambda_c(\sqrt{\varepsilon})$.\index{s@$\mathcal{C}_{\varepsilon}^a$ and $\mathcal{C}_{\varepsilon}^r$| {Theorem~\ref{Thm:class_canard1}, }}
\end{theorem}

\begin{theorem}\label{Thm:class_canard2}\cite[Theorems~4.1]{KruSzm2}
	Suppose $\varepsilon_0$, $\lambda_0$ are sufficiently small and consider $V=V_{\varepsilon_0}$ as defined in Definition~\ref{Def:nbhd}.
	Fix $\varepsilon\in (0,\varepsilon_0]$. Then the following statements hold:
	
	(i)		For $\lambda\in (-\lambda_0,\lambda_H(\sqrt{\varepsilon})]$ all orbits starting in $V$ converge to $p_e$ or leave $V$.
	
	(ii)	There exists a curve $\lambda=\lambda_{sc}(\sqrt{\varepsilon})$\index{t@$\lambda_{sc}(\sqrt{\varepsilon})$| {Theorem~\ref{Thm:class_canard2}, }} and a constant $K>0$, with
	$$
	0<\lambda_{c}(\sqrt{\varepsilon}) - \lambda_{sc}(\sqrt{\varepsilon}) = \mathcal{O}(e^{-\frac{K}{\varepsilon}}),
	$$
	such that for each $\lambda\in (\lambda_{H}(\sqrt{\varepsilon}),\lambda_{sc}(\sqrt{\varepsilon}))$ the system \eqref{Eq:class_blowup_x}--\eqref{Eq:class_blowup_lamb} has a unique attracting limit cycle $\Gamma_{(\lambda,\varepsilon)}$\index{u@$\Gamma_{(\lambda,\varepsilon)}$| {Theorem~\ref{Thm:class_canard2}, }} contained in $V$. All orbits starting in $V$, except for $p_e$, either leave $V$ or are attracted to $\Gamma_{(\lambda,\varepsilon)}$.
	
	(iii)	For $\lambda\in (\lambda_{sc},\lambda_0]$ all orbits starting in $V$, except for $p_e$, leave $V$.
\end{theorem}

\section{Introduction to the canard case}\label{Sec:canard}

As for the fold, we now consider a two-dimensional critical manifold. While the dynamics for the fold remained relatively simple, more involved phenomena appear in the canard case.
Consider the following fast-slow system:
\index{v@new canard| {\eqref{Eq:blowup_x}--\eqref{Eq:blowup_lamb}, }}
\begin{alignat}{2}
x' &= \left\{\begin{matrix}
-y + x^2&& \text{for } y< x^2,\\
0&& \text{for } y\geq x^2,
\end{matrix}\right.\label{Eq:blowup_x}\\
y' &= \varepsilon g(x,y,\lambda),\label{Eq:blowup_y}\\
\varepsilon' &= 0 ,\label{Eq:blowup_eps}\\
\lambda' &= 0,\label{Eq:blowup_lamb}
\end{alignat}
where again $g=xg_1- \lambda g_2 + yg_3$ is $\mathrm{C}^r$-smooth for $r\geq 3$, $g_i=g_i(x,y,\lambda)$ for $i=1,\ldots,3$, and with $g_1,g_2 = 1+\mathcal{O}(x,y,\lambda)$.
Also in this case we denote $a_1=\partial_x g_1(0,0,0), a_2=g_3(0,0,0)$ and assume $a_1,a_2>0$.
Again, $\varepsilon\in (0,\varepsilon_0]$ and $\lambda\in [-\lambda_0,\lambda_0]$ for certain values $\varepsilon_0,\lambda_0 >0$.
The critical manifold is given by the two-dimensional set $\mathcal{C}_0=\{(x,y):\, y\geq x^2\}$, with lower boundary $\mathcal{C}_{\partial}=\{(x,y):\, y= x^2\}$, see Figure~\ref{fig:02}.

\begin{figure}
	[htbp]
	\centering
	\begin{overpic}[scale=0.3]{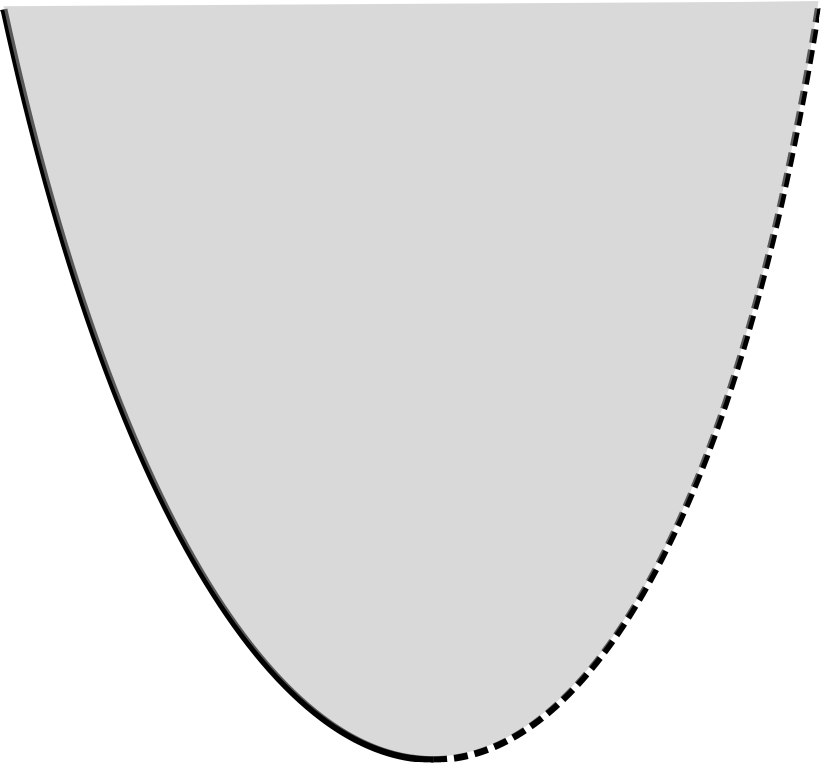}
		\put(50,40){\scalebox{1.0}{$\mathcal{C}_0$}}
		\put(8,25){\scalebox{1.0}{$\mathcal{C}_{0}^a$}}
		\put(86,25){\scalebox{1.0}{$\mathcal{C}_{0}^r$}}
		\put(36,-15){
			\begin{tikzpicture}
			\filldraw (0,0) circle (1pt) node[align=left,   below]{\scalebox{1.0}{$(0,0)$}};
			\end{tikzpicture}
		}
		\put(-40,-5){
			\begin{tikzpicture}
			\draw[->] (-1,0) -- (0,0) node[anchor = north]{$x$};
			\draw[->] (-1,0) -- (-1,1) node[anchor = west]{$y$};
			\end{tikzpicture}
		}
	\end{overpic}
	
	\caption{\\
		\label{fig:02}The two-dimensional critical manifold $\mathcal{C}_0=\{(x,y)\in \mathbb{R}^2: y\geq x^2\}$ is depicted in gray. Its lower boundary $\mathcal{C}_{\partial}$ is divided into the attracting branch $\mathcal{C}_{0}^a=\{(x,y)\in \mathbb{R}^-\times \mathbb{R}: y= x^2\}$ (solid black) and the repelling branch $\mathcal{C}_{0}^r=\{(x,y)\in \mathbb{R}^+\times \mathbb{R}: y= x^2\}$ (dashed black).
		The canard point is located at $(0,0)$.
	}
\end{figure}

\begin{definition}
	In the following, we will frequently compare the relative position of two geometric sets $A,B\subset\mathbb{R}^2$ in the plane. 
	We say that set $A$ is (locally) located to the left of set $B$ if for any $c\in \mathbb{R}$:
	$$
	a\leq b,\quad \forall a\in A\cap \{(x,c):\, x\in \mathbb{R}\},\, b\in B\cap \{(x,c):\, x\in \mathbb{R}\}.
	$$
	In this case, $B$ is (locally) located to the right of $A$. We say that $A$ is totally located to the left of $B$ if there exists $x_{\mathrm{sep}}$ such that $\{(x_{\mathrm{sep}},y):\, y\in \mathbb{R}\}$ is locally located to the right of $A$ and to the left of $B$.
	Similarly, we say that $A$ lies (locally) below $B$ if for any $c\in \mathbb{R}$:
	$$
	a\leq b,\quad \forall a\in A\cap \{(c,y):\, y\in \mathbb{R}\},\, b\in B\cap \{(c,y):\, y\in \mathbb{R}\}.
	$$
	In this case, $B$ is (locally) located above $A$. We say that $A$ is totally located below $B$ if there exists $y_{\mathrm{sep}}$ such that $\{(x,y_{\mathrm{sep}}):\, x\in \mathbb{R}\}$ is locally located above $A$ and below $B$.
	If not further specified, we always mean the local definition, see Figure~\ref{Fig:RelativePosition}.
\end{definition}

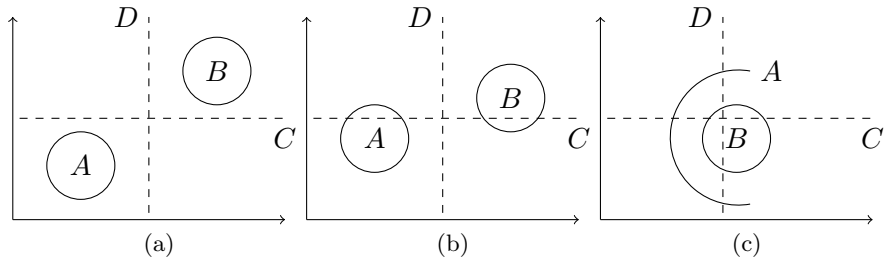
\begin{figure}[H]
	\centering
	
	\subfloat[]
	{\begin{tikzpicture}[scale=0.90]
		\draw[->] (0,0) -- (4,0) coordinate (x axis);
		\draw[->] (0,0) -- (0,3) coordinate (y axis);
		\draw (1,0.8) circle [radius=0.5] node[] {$A$};
		\draw (3,2.2) circle [radius=0.5] node[] {$B$};
		\draw[dashed] (0.1,1.5) -- (4,1.5) node[anchor=north, pos=1]{$C$};
		\draw[dashed] (2,0.1) -- (2,3) node[anchor=east, pos=1]{$D$};
		\end{tikzpicture}}
	\subfloat[]
	{\begin{tikzpicture}[scale=0.90]
		\draw[->] (0,0) -- (4,0) coordinate (x axis);
		\draw[->] (0,0) -- (0,3) coordinate (y axis);
		\draw (1,1.2) circle [radius=0.5] node[] {$A$};
		\draw (3,1.8) circle [radius=0.5] node[] {$B$};
		\draw[dashed] (0.1,1.5) -- (4,1.5) node[anchor=north, pos=1]{$C$};
		\draw[dashed] (2,0.1) -- (2,3) node[anchor=east, pos=1]{$D$};
		\end{tikzpicture}}
	\subfloat[]
	{\begin{tikzpicture}[scale=0.90]
		\draw[->] (0,0) -- (4,0) coordinate (x axis);
		\draw[->] (0,0) -- (0,3) coordinate (y axis);
		\draw (2,1.2) circle [radius=0.5] node[] {$B$};
		\draw (2.2,2.2) arc (80:280:10mm) node[anchor=west, pos=0]{$A$};
		\draw[dashed] (0.1,1.5) -- (4,1.5) node[anchor=north, pos=1]{$C$};
		\draw[dashed] (1.8,0.1) -- (1.8,3) node[anchor=east, pos=1]{$D$};
		\end{tikzpicture}}
	\caption{\label{Fig:RelativePosition}(a): Set $A$ is locally located to the left and to the right of set $B$, since there exists no horizontal line which intersects both sets, i.e. the dashed horizontal line $C$ separates $A$ and $B$. Similarly, $A$ is locally located above and below $B$, since there exists a separating vertical line $D$ (dashed). However, totally, set $A$ is located to the left of set $B$ but not to the right, since $A$ lies locally on the left side of the separating line $D$, while $B$ lies locally on the right side of $D$. Similarly, set $A$ is totally located below set $B$ but not above.
		(b): Set $A$ is locally located to the left of set $B$ but not to the right, since for example the dashed line $C$ intersects $A$ and $B$ and the $x$-component of all intersection points with $A$ are strictly smaller than those with $B$.
		At the same time, $A$ is still also totally located to the left of set $B$ and not to the right, since $A$ lies locally on the left side of the separating dashed line $D$, while $B$ lies locally on the right side of $D$.
		As in (a), $A$ is locally below and above $B$, since the vertical line $D$ separates both sets.
		However, totally, $A$ lies neither above nor below $B$, because there exists no horizontal line which lies locally above (below) $A$ and below (above) $B$.
		(c): $A$ lies locally to the left of $B$ but not totally. At the same time, $A$ lies neither locally nor totally below or above $B$.}
\end{figure}
For fixed $\varepsilon$ and all $\lambda\in [-\lambda_0,\lambda_0]$, in contrast to the unique equilibrium $p_e$ in Theorem~\ref{Thm:class_canard1}, it turns out that system \eqref{Eq:blowup_x}--\eqref{Eq:blowup_lamb} has a unique curve $\Gamma_e=\Gamma_{e,\lambda}$ of
equilibria emanating from $p_e$. This curve can locally be written as the graph of a function $(x,u_e(x))$.
Since $x\mapsto u_e(x)$ is not necessarily one-to-one, the concrete curvature of $g$ around $\Gamma_e$ is crucial for the behaviour of system \eqref{Eq:blowup_x}--\eqref{Eq:blowup_lamb} close to $\Gamma_e$. Therefore, we consider a small neighbourhood $U^{0}=U^{0}(\varepsilon,\lambda)$ around $\Gamma_e$.
We will show the existence of some $\lambda_*\in (\lambda_{H},\lambda_{sc})$ of the following kind: For $\lambda<\lambda_*$, $U^{0}$ can be chosen such that in $V_{2,\varepsilon}$, the left boundary of $U^{0}$ has distance of order $\mathcal{O}(\varepsilon +\sqrt{\varepsilon}\lambda)$ to $\Gamma_e$, while the right boundary of $U^{0}$ has distance of order $\mathcal{O}(\varepsilon^{3/2}+{\sqrt{\varepsilon}}\lambda)$ to $\Gamma_e$.
For $\lambda>\lambda_*$, $U^{0}$ has width of order $\mathcal{O}(\varepsilon^{3/2}+{\sqrt{\varepsilon}}\lambda)$ in $V_{2,\varepsilon}$.
For all $\lambda\in [-\lambda_0,\lambda_0]$, $U^{0}$ has width of order $\mathcal{O}((u_e(x))^2 + \lambda u_e(x))$ for $(x,u_e(x))\in V_{1,\varepsilon}$, see Section~\ref{Sec:Domains} for the concrete definition.
We will only consider the dynamics outside of $U^{0}$.
It will be crucial in several proofs that the boundaries of $U^{0}$ are right-curved.
If this is the case, we call $U^{0}$ right-curved.
We also write $U_0^{-},U_0^{+}$ for the left and right parts of $(V\cap \mathcal{C}_0)\backslash U^{0}$, see Section~\ref{Sec:Domains} for the concrete definition and compare Figure~\ref{Fig:LargeLambda}.
Furthermore, the slow manifold is not longer given by a one-dimensional set, but has codimension zero just as the critical manifold. For this reason, we denote by $\mathcal{C}_{\varepsilon}^a$ and $\mathcal{C}_{\varepsilon}^r$ those parts of the attracting and repelling branches of the slow manifold of \eqref{Eq:class_blowup_x}--\eqref{Eq:class_blowup_lamb} which are located below $\mathcal{C}_0$.

The adaption of Theorems~\ref{Thm:class_canard1}--\ref{Thm:class_canard2} to the system \eqref{Eq:blowup_x}--\eqref{Eq:blowup_lamb} is the following, see also Figure~\ref{Fig:numerics}:
\begin{theorem}\label{Thm:canard1}
	Suppose $\varepsilon_0$, $\lambda_0$ are sufficiently small and consider $V=V_{\varepsilon_0}$ as defined in Definition~\ref{Def:nbhd}.
	Then for all fixed $\varepsilon\in (0,\varepsilon_0]$, system \eqref{Eq:class_blowup_x}--\eqref{Eq:class_blowup_lamb} has exactly one equilibrium $p_e\in V\cap\mathcal{C}_{\partial}$, which converges to the canard point as $(\varepsilon,\lambda)\rightarrow 0$.
	Moreover, there exists a unique curve of equilibria $\Gamma_e=\Gamma_{e,\lambda}=\{(x,u_e(x))\}$\index{w@$\Gamma_e=\Gamma_{e,\lambda}$| {Theorem~\ref{Thm:canard1}, }} \index{x@$x\mapsto u_e(x)$| {Theorem~\ref{Thm:canard1}, }} in $V\cap \mathcal{C}_0$, emanating from $p_e$, together with a right-curved neighbourhood $U^{0}=U^{0}(\varepsilon,\lambda)$ around $\Gamma_e$. There exists a curve $\lambda_H(\sqrt{\varepsilon}) = -\frac{a_2}{2}\varepsilon + \mathcal{O}(\varepsilon^{\frac{3}{2}})$ such that for each value $\lambda\in (-\lambda_0,\lambda_H(\sqrt{\varepsilon})]$ all trajectories starting in $U_0^{-}\cup U_0^{+}$ leave $V$ in $U_0^{+}$.
	In particular, all solutions starting in $U_0^{-}$ leave $\mathcal{C}_0$, surround $\Gamma_e$ close to $p_e$, return to $\mathcal{C}_0$ in $U_0^{+}$ and leave $V$ in $U_0^{+}$.
	
	There exists a smooth function 
	$$\lambda_c(\sqrt{\varepsilon})=-\left(\frac{a_2}{2} + \frac{-2a_1 -2a_2}{8}\right)\varepsilon + \mathcal{O}(\varepsilon^{3/2})=\frac{a_1 - a_2}{4}\varepsilon + \mathcal{O}(\varepsilon^{3/2}),$$
	such that $\mathcal{C}_{\varepsilon}^a$ connects to $\mathcal{C}_{\varepsilon}^r$ if and only if $\lambda=\lambda_c(\sqrt{\varepsilon})$. Moreover, $\mathcal{C}_{\varepsilon}^a$ and $\mathcal{C}_{\varepsilon}^r$ do not intersect $\mathcal{C}_0$ for $\lambda=\lambda_c(\sqrt{\varepsilon})$.
	In particular, maximal canard solutions exist if and only if $\lambda=\lambda_c(\sqrt{\varepsilon})$, and those solutions are located below $\mathcal{C}_0$ in $V$.
\end{theorem}

Theorem~\ref{Thm:class_canard2} claims the existence of limit cycles $\Gamma_{(\lambda,\varepsilon)}$ for each $\lambda\in (\lambda_{H}(\sqrt{\varepsilon}),\lambda_{sc}(\sqrt{\varepsilon}))$. 
In system \eqref{Eq:blowup_x}--\eqref{Eq:blowup_lamb}, the same orbits are observed but only below $\mathcal{C}_0$.
Therefore, we write $\tilde{\Gamma}_{(\lambda,\varepsilon)}:= (\mathbb{R}^2\backslash\mathcal{C}_0)\cap\Gamma_{(\lambda,\varepsilon)}$
\index{y@$\tilde{\Gamma}_{(\lambda,\varepsilon)}$| {after Theorem~\ref{Thm:canard1}, }} for the part of $\Gamma_{(\lambda,\varepsilon)}$ which is located below $\mathcal{C}_0$.
Since $x'=0$ in $\mathcal{C}_0$, each limit cycle of the classical canard system perturbs to the half orbit $\tilde{\Gamma}_{(\lambda,\varepsilon)}$ and its vertical extensions.
More precisely, consider the left and right intersection points $p_-=(p_-^x,p_-^y)$ and $p_+=(p_+^x,p_+^y)$ in $\Gamma_{(\lambda,\varepsilon)}\cap \mathcal{C}_{\partial}$, i.e. $p_-^x<p_+^x$.
Moreover, denote by
\begin{equation}\label{Def:P_-,P_+}
P_-:= \{(p_-^x,y):\, y\in \mathbb{R}\},\qquad P_+:= \{(p_+^x,y):\, y\in \mathbb{R}\}\index{z@$p_-, p_+, P_-, P_+$| {\eqref{Def:P_-,P_+}, }}
\end{equation}
the vertical extensions of $\Gamma_{(\lambda,\varepsilon)}$.
Then in system \eqref{Eq:blowup_x}--\eqref{Eq:blowup_lamb}, the limit cycle $\Gamma_{(\lambda,\varepsilon)}$ can be rediscovered in the form $$\tilde{\Gamma}_{(\lambda,\varepsilon)}\cup [\mathcal{C}_0\cap (P_-\cup P_+)].$$
The canard explosion for system \eqref{Eq:blowup_x}--\eqref{Eq:blowup_lamb} looks as follows, see also Figure~\ref{Fig:numerics}:

\begin{theorem}\label{Thm:canard2}
	Suppose $\varepsilon_0$, $\lambda_0$ are sufficiently small and consider $V=V_{\varepsilon_0}$ as defined in Definition~\ref{Def:nbhd}. Fix $\varepsilon\in (0,\varepsilon_0]$ and consider the notation from Theorem~\ref{Thm:canard1}. Then the following statements hold:
	
	(i)		For $\lambda\in (-\lambda_0,\lambda_H(\sqrt{\varepsilon})]$ all trajectories starting in $U_0^{-}\cup U_0^{+}$ leave $V$ in $U_0^{+}$.
	Moreover, all solutions starting in $U_0^{-}$ leave $\mathcal{C}_0$, surround $\Gamma_e$ close to $p_e$, return to $\mathcal{C}_0$ in $U_0^{+}$ and leave $V$ in $U_0^{+}$.
	
	(ii)	There exists a curve $\lambda=\lambda_{sc}(\sqrt{\varepsilon})$ and a constant $K>0$, with
	$$
	0<\lambda_{c}(\sqrt{\varepsilon}) - \lambda_{sc}(\sqrt{\varepsilon}) = \mathcal{O}(e^{-\frac{K}{\varepsilon}}),
	$$
	such that for each $\lambda\in (\lambda_{H}(\sqrt{\varepsilon}),\lambda_{sc}(\sqrt{\varepsilon}))$, system \eqref{Eq:blowup_x}--\eqref{Eq:blowup_lamb} has a unique attracting half cycle $\tilde{\Gamma}_{(\lambda,\varepsilon)}$ contained in $V\backslash\mathcal{C}_0$ and extended into $\mathcal{C}_0$ by $\mathcal{C}_0\cap P_-$ and $\mathcal{C}_0\cap P_+$. Moreover, all solutions of \eqref{Eq:blowup_x}--\eqref{Eq:blowup_lamb} starting in $V\backslash U^{0}$ and to the left of $P_-$ or below $\tilde{\Gamma}_{(\lambda,\varepsilon)}$ leave $V$ to the total right of $P_+$. All trajectories starting in $U_0^{-}$ and to the left of $P_-$ leave $V$ to the total right of $P_+$ and in $U_0^{+}$. Finally, all solutions starting in $V\backslash U^{0}$, between $P_-$ and $P_+$ and above $\tilde{\Gamma}_{(\lambda,\varepsilon)}$ leave $V$ in $U_0^{+}$ and to the total left of $P_+$.
	
	(iii)	For $\lambda\in (\lambda_{sc},\lambda_0]$, all orbits starting in $V\backslash U^{0}$, leave $V$.
	
	(iv) There exists some $\lambda_*\in (\lambda_{H}(\sqrt{\varepsilon}),\lambda_{sc}(\sqrt{\varepsilon}))$ such that for $\lambda<\lambda_*$, $U^{0}\cap V_{2,\varepsilon}$ has width $\mathcal{O}(\varepsilon +\sqrt{\varepsilon}\lambda)$, while
	for $\lambda>\lambda_*$, $U^{0}\cap V_{2,\varepsilon}$ has width of order $\mathcal{O}(\varepsilon^{3/2}+{\sqrt{\varepsilon}}\lambda)$.
	For all $\lambda\in [-\lambda_0,\lambda_0]$, $U^{0}\cap V_{1,\varepsilon}$ has width of order $\mathcal{O}((u_e(x))^2 + \lambda u_e(x))$ for $(x,u_e(x))\in V_{1,\varepsilon}$.
\end{theorem}

\begin{remark}
	Note that the adaption of Theorem~\ref{Thm:class_canard2}.iii to system \eqref{Eq:blowup_x}--\eqref{Eq:blowup_lamb} is more involved. Indeed, different phenomena can appear.
	One first naive idea for the corresponding behaviour would be the following:
	
	(iii)	For $\lambda\in (\lambda_{c},\lambda_0]$ all orbits starting in $V\backslash U^{0}$ leave $V$ in $V\backslash\mathcal{C}_0$ and to the total right of 
	$\{(p_e^x,y):\, y\in \mathbb{R}\}$.
	
	This would reflect the fact that in the classical canard case, for $\lambda\in (\lambda_{sc},\lambda_0]$, trajectories eventually end up on the right of the repelling slow manifold $\mathcal{C}_{\varepsilon}^r$ and are being repelled fast towards the right in $(x,y)$-phase space.
	However, it can happen that some solutions first spiral around the equilibrium $p_e$ before this happens.
	In system \eqref{Eq:blowup_x}--\eqref{Eq:blowup_lamb}, such solutions enter $U_0^{+}$ already before finishing the first half spiral and continue vertically upwards in $U_0^{+}$ until they leave $V$.
\end{remark}

An extension of Theorem~\ref{Thm:canard2}.iii is the following, see Figure~\ref{Fig:LargeLambda}
\begin{theorem}\label{Thm:canard3}
	Suppose $\varepsilon_0$, $\lambda_0$ are sufficiently small and consider $V=V_{\varepsilon_0}$ as defined in Definition~\ref{Def:nbhd}. Fix $\varepsilon\in (0,\varepsilon_0]$ and consider the notation from Theorems~\ref{Thm:canard1}--\ref{Thm:canard2}.
	
	For $\lambda\in (\lambda_{c},\lambda_0]$, there exists a vertical line $P_c=P_{c,\lambda}=\{(p_c^x,y):\, y\in \mathbb{R}\}$\index{za@$P_c=P_{c,\lambda}$| {Theorem~\ref{Thm:canard3} ,}} such that the following holds true: 
	All solutions starting in $V\backslash U^{0}$ and to the left of $P_c$ leave $V$ in $V\backslash\mathcal{C}_0$.
	All orbits starting in $\mathcal{C}_0\backslash U^{0}$ and to the right of $P_c$ leave $V$ in $U_0^{+}$.
	
	Note that the set of solutions starting in $V\backslash U^{0}$ and to the left of $P_c$ can be empty, since $P_c$ can be located to the left of $V$.
\end{theorem}

\begin{remark}
	Note that in Theorem~\ref{Thm:canard3}, among all trajectories starting to the right of $P_c$, besides orbits which leave $V$ in $U_0^{+}$, some solutions might leave $V$ in $V\backslash\mathcal{C}_0$. This is why we have to restrict to orbits starting in $\mathcal{C}_0\backslash U^{0}$, rather than considering initial conditions in the larger set $V\backslash U^{0}$.
	
	Indeed, some orbits starting in $V\backslash\mathcal{C}_0$ and to the right of $P_c$ correspond to orbits which spiral around $p_e$ in the classical case since they are trapped by $\mathcal{C}_{\varepsilon}^r$, and some others are located below and to the right of $\mathcal{C}_{\varepsilon}^r$ and consequently move constantly to the right in $(x,y)$-phase space.
	In \eqref{Eq:blowup_x}--\eqref{Eq:blowup_lamb}, orbits of the first kind leave $V$ in $U_0^{+}$, while orbits of the second kind leave $V$ in $V\backslash\mathcal{C}_0$, see Figure~\ref{Fig:LargeLambda}.
\end{remark}

\begin{figure}[H]
	\centering	
	\subfloat[]
	{	
}
\end{figure}
\begin{figure}
	\ContinuedFloat
	\centering
	\caption{In this numerical example, $g=x(1+x)^{a_1} - g_2\lambda + g_3y$, with $a_1=1=g_2$, $g_3=a_2=0.9$ and $\varepsilon=0.01$.
		In particular, $\lambda_H \approx \frac{-a_2}{2}\varepsilon = -4.5\cdot 10^{-3}=:\lambda_{H,\mathrm{appr}}$ and $\lambda_c \approx \frac{a_1-a_2}{4}\varepsilon = 2.5\cdot 10^{-4}=:\lambda_{c,\mathrm{appr}}$. The value of $\lambda$ increases along the pictures (a)-(f). In (a), $\lambda=1.5\cdot\lambda_{H,\mathrm{appr}}<\lambda_H$, and we have an attracting equilibrium in the classical canard case (dashed line), while the trajectory enters $\mathcal{C}_0$ in $U^+$ in the non-smooth case (solid line).
		In (b),(c), $\lambda=0.5\cdot\lambda_{H,\mathrm{appr}}>\lambda_H$, and we have a small attracting periodic orbit in the classical canard case (dashed line), while the trajectory enters $\mathcal{C}_0$ in $U^+$ in the non-smooth case (solid line). In (b), the trajectory for the non-smooth system starts above the exterior of the periodic orbit, while it starts above the interior of the periodic orbit in (c). As stated in Theorem~\ref{Thm:canard2}, the trajectory leaves $V$ to the right of the periodic orbit in (b), and the solution in (c) reenters $\mathcal{C}_0$ inside the orbit, i.e. it leaves $V$ between $P_-,P_+$.
		In (d)-(e), $\lambda=0.2\cdot\lambda_{c,\mathrm{appr}}<\lambda_{sc}$, and we have a larger attracting periodic orbit in the classical canard case (dashed line), while the trajectory again reenters $\mathcal{C}_0$ in $U^+$ in the non-smooth case, but further to the right (solid line).
		Moreover, we observe the same relative positioning of the solution to the non-smooth system and the periodic orbit as in $(b),(c)$.
		In (e), $\lambda=\lambda_{c,\mathrm{appr}}$, and we obtain trajectories close to the maximal canard solution in the classical canard case (dashed line), as well as in the non-smooth case (solid line), once those solutions cross the parabola $\mathcal{C}_{\partial}$.}\label{Fig:numerics}
\end{figure}
\begin{figure}[H]
	\centering
	\begin{overpic}[width=0.4\textwidth]{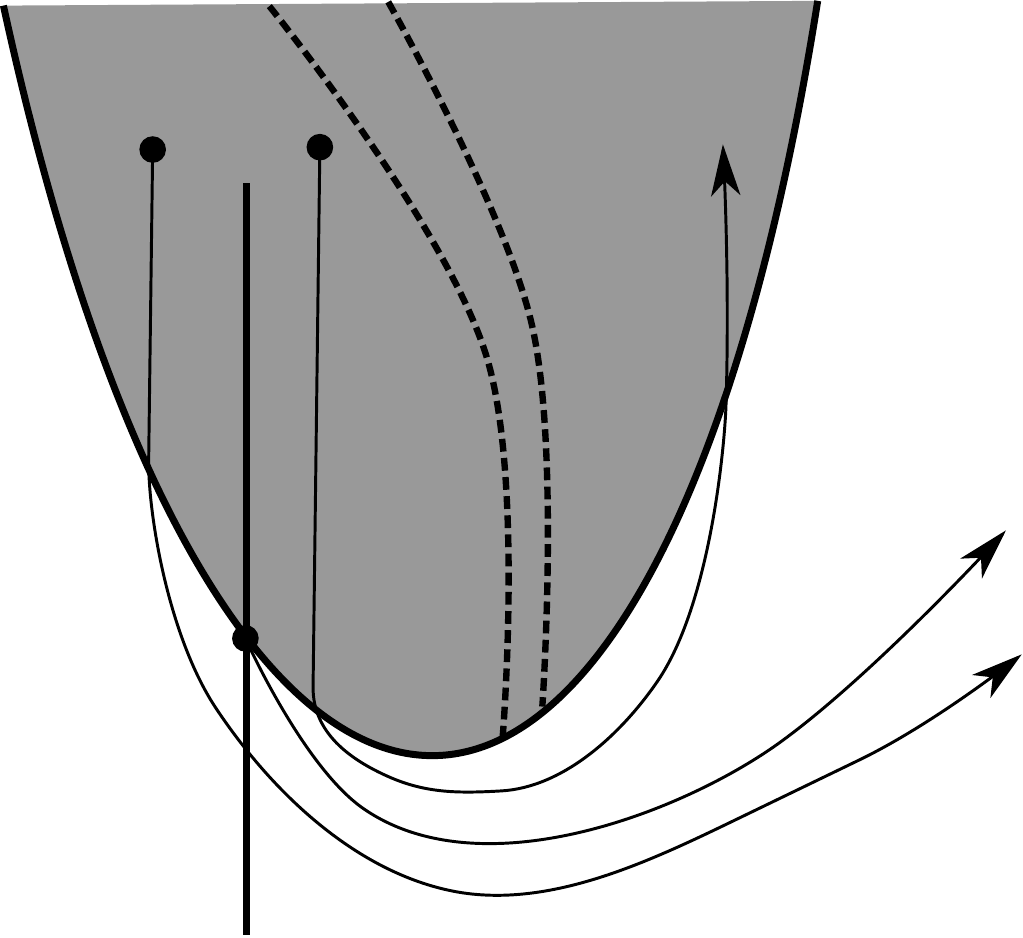}
		\put(-7,83){\scalebox{1.0}{$\mathcal{C}_{\partial}$}}
		\put(15,0){\scalebox{1.0}{$P_c$}}
		\put(32,84){\scalebox{0.8}{$U_0^{0}$}}
		\put(15,83){\scalebox{1.0}{$U_0^{-}$}}
		\put(55,83){\scalebox{1.0}{$U_0^{+}$}}
		\put(-40,-5){
			\begin{tikzpicture}
			\draw[->] (-1,0) -- (0,0) node[anchor = north]{$x$};
			\draw[->] (-1,0) -- (-1,1) node[anchor = west]{$y$};
			\end{tikzpicture}
		}
	\end{overpic}
	\caption{\label{Fig:LargeLambda}
		For $\lambda\in (\lambda_{c}(\sqrt{\varepsilon}),\lambda_0]$, all solutions starting in $P_c\cap \mathcal{C}_0$ leave $V$ in $V\backslash\mathcal{C}_0$.
		In particular, this holds for the trajectory which starts at $(p_c^x,p_c^y)=P_c\cap \mathcal{C}_{\partial}$ (starting point in the middle).
		Trajectories starting in $U_0^{-}$ (gray area left to the dashed lines) and to the left of $P_c$ (left starting point) also leave $V$ in $V\backslash\mathcal{C}_0$.
		Trajectories starting in $U_0^{-}$ and to the right of $P_c$ (right starting point) leave $V$ in $U_0^{+}$ (gray area right of the dashed lines). The set $\{g=0\}\cap \mathcal{C}_0$ is contained in $U_0^{0}$ (gray area between the dashed lines).
	}
\end{figure}
As for Theorems~\ref{Thm:class_canard1}--\ref{Thm:class_canard2}, we will show Theorems~\ref{Thm:canard1}--\ref{Thm:canard3} with help of the weighted polar blow-up transformation $\Psi$ as defined in \eqref{Blow_up_Psi} and the directional charts $K_1$ and $K_2$ according to \eqref{Blow_up_Phi1}--\eqref{Blow_up_Phi2}. In particular, the desingularized vector field in chart $K_1$ is given by
\index{zb@new canard in $K_1$| {\eqref{Eq:blowup_x_1}--\eqref{Eq:blowup_lamb_1}, }}
\begin{alignat}{2}
x_1' &= \left\{\begin{matrix}
-1 + x_1^2 - \frac{1}{2}\varepsilon_1 x_1 F(r_1,x_1,\varepsilon_1,\lambda_1),&& \text{for } 1< |x_1|,\\
-\frac{1}{2}\varepsilon_1 x_1 F(r_1,x_1,\varepsilon_1,\lambda_1),&& \text{for } 1\geq |x_1|,
\end{matrix}\right.\label{Eq:blowup_x_1}\\
r_1' &= \frac{1}{2}r_1\varepsilon_1 F(r_1,x_1,\varepsilon_1,\lambda_1),\label{Eq:blowup_r_1}\\
\varepsilon_1' &= -\varepsilon_1^2 F(r_1,x_1,\varepsilon_1,\lambda_1) ,\label{Eq:blowup_eps_1}\\
\lambda_1' &= -\frac{1}{2}\lambda_1\varepsilon_1 F(r_1,x_1,\varepsilon_1,\lambda_1),\label{Eq:blowup_lamb_1}
\end{alignat}
where $F(r_1,x_1,\varepsilon_1,\lambda_1)= x_1 - \lambda_1 + r_1(a_1x_1^2 + a_2) + \mathcal{O}(r_1(|r_1|+|\lambda_1|))$.
\index{zcb@$F(r_1,x_1,\varepsilon_1,\lambda_1)$| {\eqref{Eq:blowup_x_1}--\eqref{Eq:blowup_lamb_1}, }}
Similarly, the desingularized vector field in the rescaling chart $K_2$ is given by
\index{zd@new canard in $K_2$| {\eqref{Eq:blowup_x_2}--\eqref{Eq:blowup_y_2}, }}
\begin{alignat}{2}
x_2' &= \left\{\begin{matrix}
- y_2 + x_2^2,&& \text{for } y_2< x_2^2,\\
0,&& \text{for } y_2\geq x_2^2,
\end{matrix}\right.\label{Eq:blowup_x_2}\\
y_2' &= x_2 - \lambda_2 + r_2 G(x_2,y_2) + \mathcal{O}(r_2(|\lambda_2| + r_2)),\label{Eq:blowup_y_2}
\end{alignat}
where $G(x_2,y_2) = a_1x_2^2 + a_2y_2$.
\index{ze@$G(x_2,y_2)$| {\eqref{Eq:blowup_x_2}--\eqref{Eq:blowup_y_2}, }}

\subsection{First chart}
We analyze system \eqref{Eq:blowup_x_1}--\eqref{Eq:blowup_lamb_1} for $\lambda_1\in (-\mu,\mu)$ with $\mu>0$ small.
In particular, for fixed $\varepsilon$, we prove that trajectories starting in $V_{1,\varepsilon}$ (see Section~\ref{Sec:Intro}) and to the left of the set $U_{0,1}^{1}=\Phi_1^{-1}(U_0^{0},\varepsilon,\lambda)\cap V_1$ reach the domain $V_{2,\varepsilon}$ of transformation $\Phi_2$ in finite time. On the other hand, solutions of system \eqref{Eq:blowup_x_1}--\eqref{Eq:blowup_lamb_1} which start to the right of $U_{0,1}^{1}$ move away from the canard point and from $V_{2,\varepsilon}$.
Since system \eqref{Eq:blowup_x_1}--\eqref{Eq:blowup_lamb_1} is only piecewise smooth, we have to split $V_{1,\varepsilon}$ into the sets $\{x_1<-1\}$, $\{-1\leq x_1<0\}$, $\{0<x_1\leq 1\}$ and $\{1< x_1\}$, and study the dynamics in each of these sets separately.
Note that the hyperplanes $\{r_1=0\}$, $\{\varepsilon_1=0\}$ and $\{\lambda_1=0\}$ are invariant. Moreover, the line $l_1=\{(x_1,0,0,0):x_1\in \mathbb{R}\}$ contains the segment of equilibria $\{(x_1,0,0,0):x_1\in [-1,1]\}$, with endpoints $p_a=(-1,0,0,0)$ and $p_r=(1,0,0,0)$\index{zf@$p_a, p_r$}.
For the flow on the line $l_1$, $p_a$ is attracting from the direction $x_1<-1$ and $p_r$ is repelling from $x_1>1$.
The blown-up left branch of $\mathcal{C}_{\partial}$ in chart $K_1$ is given by $\mathcal{C}_{0,1}^a=\{x_1=-1, r_1\geq 0\}$. Similarly, the right branch is given by $\mathcal{C}_{0,1}^r=\{x_1=1, r_1\geq 0\}$ \index{zg@$\mathcal{C}_{0,1}^a$ and $\mathcal{C}_{0,1}^r$}.

\begin{figure}
	[htbp]
	\centering
	
	\caption{\\
		\label{Fig:03}Projection to $\lambda_1=\varepsilon_1=0$. Critical manifold $\mathcal{C}_{0,1}$ (gray) in chart $K_1$ with blown-up attracting branch $\mathcal{C}_{0,1}^a$ and blown-up repelling branch $\mathcal{C}_{0,1}^r$. Invariant line $l_1$ (dashed) with segment of equilibria $[p_a,p_r]$, where $p_a$ is attracting from the left and $p_r$ is repelling towards the right.
		The set $\{F=0\}\cap \mathcal{C}_{0,1}$ is located in $U_0^{1}$ (area between the dashed curves).}
\end{figure}

In particular, the endpoints of $\mathcal{C}_{0,1}^a, \mathcal{C}_{0,1}^r$ are given by $p_a$ and $p_r$ respectively. 
In the invariant subset $\{\varepsilon_1=0=\lambda_1\}$, the sets $\mathcal{C}_{0,1}^a, \mathcal{C}_{0,1}^r$ correspond to equilibria on the boundary of $\mathcal{C}_{0,1}$ of the reduced system 
\begin{alignat*}{2}
x_1' &= \left\{\begin{matrix}
-1 + x_1^2,&& \text{for } 1< |x_1|,\\
0,&& \text{for } 1\geq |x_1|,
\end{matrix}\right.\\
r_1' &= 0.
\end{alignat*}
The other equilibria of this reduced system are given by the set $\{(x_1,r_1,0,0): |x_1|<1, r_1\geq 0 \}$, i.e. by the projection of $\mathrm{int}(\mathcal{C}_{0,1})$ to $\{\varepsilon_1=0=\lambda_1\}$, see Figure~\ref{Fig:03}.
In the invariant subset $\{r_1=0=\lambda_1\}$, the reduced system reads 
\begin{alignat*}{2}
x_1' &= \left\{\begin{matrix}
-1 + x_1^2 - \frac{1}{2}\varepsilon_1x_1^2,&& \text{for } 1< |x_1|,\\
- \frac{1}{2}\varepsilon_1x_1^2,&& \text{for } 1\geq |x_1|,
\end{matrix}\right.\\
\varepsilon_1' &= - \varepsilon_1^2x_1.
\end{alignat*}
The equilibria of this system are given by $$\{(0,0,\varepsilon_1,0): \varepsilon_1\ge 0 \}\cup\{(x_1,0,0,0): |x_1|\leq 1 \},$$
and $p_a,p_r$ are equilibria on the boundary of this set.
Since $\lambda_1,r_1$ are small, $F(r_1,x_1,\varepsilon_1,\lambda_1)= x_1 - \lambda_1 + r_1(a_1x_1^2 + a_2) + \mathcal{O}(|r_1|(|r_1|+|\lambda_1|))\neq 0$ for $|x_1|$ close to one.
Therefore, $F\neq 0$ in the vicinity of $p_a,p_r$. In particular, $F<0$ close to $p_a$ and $F>0$ close to $p_r$.
Consider the smooth subsystems
\begin{alignat*}{2}
\left\{
\begin{matrix}
x_1' = -1 + x_1^2 - \frac{1}{2}\varepsilon_1x_1^2,\\
\varepsilon_1' = - \varepsilon_1^2x_1,
\end{matrix}
\right.
\qquad
\left\{
\begin{matrix}
x_1' = - \frac{1}{2}\varepsilon_1x_1^2,\\
\varepsilon_1' = - \varepsilon_1^2x_1.
\end{matrix}\right.
\end{alignat*}
For $p\in \{(0,0,\varepsilon_1,0), \varepsilon_1\ge 0 \}\cup \{(x_1,0,0,0), |x_1|<1 \}$ we only have to consider the second subsystem since those equilibrium points are contained in the interior of the critical manifold.
For all of these $p$, zero is the only eigenvalue of the (linearized) second subsystem.
For the boundary equilibria $p_a$ and $p_r$, zero is the only eigenvalue of the (linearized) second subsystem, which corresponds to perturbations in directions $x_1>-1$ and $x_1<1$ respectively.
For the (linearized) first subsystem, i.e. for directions $x_1<-1$ and $x_1>1$ in the full system, both $p_a$ and $p_r$ have a triple zero eigenvalue with eigenvectors $(0,1,0,0),(0,0,0,1)$, and $(1,0,-2,0)$ for $p_a$ respectively $(1,0,2,0)$ for $p_r$.
Moreover, in the direction $x_1<-1$, $p_a$ has the eigenvalue $-2$, while $p_r$ has the eigenvalue $2$ in the direction $x_1>1$, both with eigenvector $(1,0,0,0)$.
In the full system \eqref{Eq:blowup_x_1}--\eqref{Eq:blowup_lamb_1} also note that $x=r_1x_1$ remains constant in $\mathcal{C}_{0,1}=\{|x_1|\leq 1\}$ as a consequence of the second equation in \eqref{Eq:blowup_x}.
Similarly, $\varepsilon=r_1^2\varepsilon_1$ and $\lambda=r_1\lambda_1$ remain constant because of \eqref{Eq:blowup_eps} and \eqref{Eq:blowup_lamb} respectively.

We can find some $C_1,C_2>0$ such that 
\begin{equation}\label{Def:U_1^1}
\begin{split}
U_{1}^{1}:= \{&(x_1,r_1 + \tilde{r}_1,\lambda_1 + \tilde{\lambda}_1,\varepsilon_1):\ x_1 - \lambda_1 + r_1(a_1x_1^2 + a_2)=0,\\
& |\tilde{r}_1| \leq C_1 r_1^2,\ |\tilde{\lambda}_1\tilde{r}_1|\leq C_2|\lambda|
\}\cap V_1
\end{split}
\index{zh@$U_1^1$| {\eqref{Def:U_1^1}, }}
\end{equation}
contains the set $\{F(x_1,r_1,\lambda_1,\varepsilon_1)=0\}\cap V_1$ of roots of $F$.
For $\lambda_1 \in [-\mu,\mu]$, $r_1 \in [-\rho,\rho]$ with $\mu,\rho$ small enough, the function
$x_1\mapsto -\frac{x_1-\lambda_1}{a_1x_1^2 + a_2}$ is decreasing for $x_1\in [-C_3,C_3]$ provided that $C_3 \in (0,1)$ is small enough, while $-\frac{x_1-\lambda_1}{a_1x_1^2 + a_2}> \rho$ for $|x_1|\in [C_3,x_{0,1}]$.
Moreover, $-\frac{x_1-\lambda_1}{a_1x_1^2 + a_2} >0$ for $x_1<-\lambda_1$.
We introduce 
\begin{equation}\label{Def:U_1^+-}
\begin{split}
U_{1}^{-}&:=\{F(x_1,r_1,\lambda_1,\varepsilon_1)<0\}\backslash U_{1}^{1},\\ \quad U_{1}^{+}&:=\{F(x_1,r_1,\lambda_1,\varepsilon_1)>0\}\backslash U_{1}^{1},\quad
U_{0,1}^{1}:=U_{1}^{1}\cap\mathcal{C}_{0,1}, \text{ and}\\
\quad U_{0,1}^{-}&:= U_{1}^{-}\cap\mathcal{C}_{0,1},\quad U_{0,1}^{+}:= U_{1}^{+}\cap\mathcal{C}_{0,1}.
\end{split}
\index{zi@$U_{1}^{-}, U_{1}^{+}, U_{0,1}^{1}, U_{0,1}^{-}, U_{0,1}^{+}$| {\eqref{Def:U_1^+-} ,}}
\end{equation}

Then the following lemma holds, see Figure~\ref{Fig:FirstChart3dim}:
\begin{lemma}\label{Lem:chart1}
	For $\rho,\mu$ small enough and $\varepsilon=r_1^2 \varepsilon_1$, $\lambda = r_1\lambda_1$, we obtain the following types of trajectories for system \eqref{Eq:blowup_x_1}--\eqref{Eq:blowup_lamb_1}:
	\begin{enumerate}
		\item[(i)] For $(x_1(0),r_1(0),\lambda_1(0),\varepsilon_1(0))\in U_{1}^{-}\backslash \mathcal{C}_{0,1}$, $x_1$ eventually reaches a small ball around $-1$ with $x_1<-1$ and hence remains in $U_{1}^{1}\backslash \mathcal{C}_{0,1}$. Moreover,
		$$
		|r_1|'<0, \quad \varepsilon_1'>0, \quad \lambda_1'\mathrm{sgn}(\lambda_1(0)) >0.
		$$
		
		\item[(ii)] For $(x_1(0),r_1(0),\lambda_1(0),\varepsilon_1(0))\in U_{0,1}^{-}$ and $-1 \leq x_1(0) <0$:
		$$
		x_1\downarrow -1,\quad r_1\rightarrow -x_1(0)r_1(0), \quad \varepsilon_1\uparrow \frac{\varepsilon}{r_1(0)^2x_1(0)^2}, \quad \lambda_1\rightarrow -\frac{\lambda}{r_1(0)x_1(0)},
		$$
		as time increases.
		Moreover, these values are attained at a finite time $\tilde{t}_1>0$.
		
		\item[(iii)] For $(x_1(0),r_1(0),\lambda_1(0),\varepsilon_1(0))\in U_{0,1}^{+}$, $r_1(0)\neq 0$ and $x_1(0) \leq 1$:
		$$
		x_1\downarrow \frac{r_1(0)x_1(0)}{\rho},\quad |r_1|\uparrow \rho, \quad \varepsilon_1\downarrow \frac{\varepsilon}{\rho^2}, \quad \lambda_1\rightarrow \frac{\lambda}{\rho},\quad \text{as }t_1 \text{ increases}.
		$$
		
		\item[(iv)] For $(x_1(0),r_1(0),\lambda_1(0),\varepsilon_1(0))\in U_{1}^{+}\backslash \mathcal{C}_{0,1}$:
		$$
		x_1\uparrow x_{1,0},\quad |r_1|\uparrow \rho, \quad \varepsilon_1\downarrow \frac{\varepsilon}{\rho^2}, \quad \lambda_1\rightarrow \frac{\lambda}{\rho},\quad \text{as }t_1 \text{ increases}.
		$$
	\end{enumerate}
\end{lemma}

\begin{figure}
	[htbp]
	\centering
	\caption{
		\label{Fig:FirstChart3dim}
		Dynamics in Chart $K_1$. The gray area depicts the projection of $\mathcal{C}_{0,1}$ to $\{\varepsilon_1=0\}\cup\{r_1=0\}$.
		The trajectories from left to right correspond to the cases (i)--(iv) in Lemma~\ref{Lem:chart1}.
	}
\end{figure}

\begin{proof}
	Statement~i follows directly from the discussion above. For Statement~ii, note that $x_1\rightarrow -1$ as time increases because of the second equation in \eqref{Eq:blowup_x_1} and since $F<0$, $x_1<0$ and $\varepsilon_1>0$ are bounded away from zero. Consequently, $x_1=-1$ at a finite time $\tilde{t}_1>0$. Moreover,  $x=r_1x_1$, $\varepsilon=r_1^2\varepsilon_1$ and $\lambda=r_1\lambda_1$ remain constant as discussed above. This implies $r_1\rightarrow -x_1(0)r_1(0)$, $\varepsilon_1\rightarrow \frac{\varepsilon_1(0)r_1(0)^2}{r_1(0)^2x_1(0)^2}=\frac{\varepsilon}{r_1(0)^2x_1(0)^2}$ and $\lambda_1\rightarrow -\frac{\lambda_1(0)r_1(0)}{r_1(0)x_1(0)}=-\frac{\lambda}{r_1(0)x_1(0)}$ for $t_1\nearrow \tilde{t}_1$.
	Statements~iii,iv are obtained similarly.
\end{proof}

\subsection{Second chart}
In this subsection, we study the chart $K_2$, i.e. \eqref{Eq:blowup_x_2}--\eqref{Eq:blowup_y_2}:
\begin{alignat*}{2}
x_2' &= \left\{\begin{matrix}
- y_2 + x_2^2,&& \text{for } y_2< x_2^2,\\
0,&& \text{for } y_2\geq x_2^2,
\end{matrix}\right.\\
y_2' &= x_2 - \lambda_2 + r_2 (a_1x_2^2 + a_2y_2) + \mathcal{O}(r_2(|\lambda_2| + r_2)).
\end{alignat*}
While the dynamics restricted to the critical manifold $\mathcal{C}_{0,2}=\{y_2\geq x_2^2\}$ \index{zj@$\mathcal{C}_{0,2}$} remains considerably easy, it turns out that solutions of \eqref{Eq:blowup_x_2}--\eqref{Eq:blowup_y_2} are obtained as small perturbations of constants of motion of the function $H(x_2,y_2)=\frac{1}{2}\exp(-2y_2)\left(y_2-x_2^2+\frac{1}{2}\right)$\index{zk@$H(x_2,y_2)=\frac{1}{2}e^{-2y_2}\left(y_2-x_2^2+\frac{1}{2}\right)$}.
To see this, consider the invariant subset $\{r_2=0=\lambda_2\}$ and the reduced system
\begin{alignat}{2}
x_2' &= \left\{\begin{matrix}
- y_2 + x_2^2,&& \text{for } y_2< x_2^2,\\
0,&& \text{for } y_2\geq x_2^2,
\end{matrix}\right.\label{Eq:const_of_mo1}\\
y_2' &= x_2.\label{Eq:const_of_mo2}
\end{alignat}

\begin{lemma}\label{Lem1:chart2}
	For system \eqref{Eq:const_of_mo1}--\eqref{Eq:const_of_mo2}, we obtain the following:
	\begin{enumerate}
		\item[(i)] For initial conditions of the form $x_2(0)=-c<0$ and $y_2(0)>x_2(0)^2$, $(x_2,y_2)$ reaches the point $(-c,c^2)$ after finite time with $x_2'=0$, $y_2'=x_2=-c$.
		From $(-c,c^2)$, $(x_2,y_2)$ continues to the point $(c,c^2)$ as a constant of motion of $H(x_2,y_2)=\frac{1}{4}\exp(-2c)$.
		Afterwards, $y_2\rightarrow \infty$ as $t_2\rightarrow\infty$ at speed $c>0$, while $x_2=c$ remains constant.
		
		\item[(ii)] For initial conditions of the form $x_2(0)=c>0$ and $y_2(0)>x_2(0)^2$, $y_2\rightarrow \infty$ as $t_2\rightarrow\infty$ at speed $c>0$, while $x_2=c$ remains constant.
		
		\item[(iii)] The set $\mathbb{R}^2\backslash\mathcal{C}_{0,2}=\{y_2\leq x_2^2\}$ is invariant for all solutions of \eqref{Eq:const_of_mo1}--\eqref{Eq:const_of_mo2} with $H(x_2(0),y_2(0))<0$. In particular, each such solution is a constant of motion of $H$ with $x_2\rightarrow \infty$ as $t_2\rightarrow\infty$.
		
		\item[(iv)] All equilibria of \eqref{Eq:const_of_mo1}--\eqref{Eq:const_of_mo2} are given by the half line $\{(0,y):\,\, y\in \mathbb{R}^+\}\subset \mathcal{C}_{0,2}$.
		
		\item[(v)] The solution $\gamma_{c,2}(t_2)= \left(\frac{1}{2}t_2,\frac{1}{4}t_2^2 - \frac{1}{2}\right)$, $t_2\in \mathbb{R}$, is a constant of motion for $H(x_2(0),y_2(0))=0$
		\index{zl@$\gamma_{c,2}(t_2)$| {Lemma~\ref{Lem1:chart2}, }}.
	\end{enumerate}
\end{lemma}

\begin{proof}
	If $(x_2(0),y_2(0))\in \{y_2<x_2^2\}$, then $(x_2,y_2)$ solves the system
	\begin{alignat*}{2}
	x_2' &=
	- y_2 + x_2^2,\\
	y_2' &= x_2.
	\end{alignat*}
	In particular, $(x_2,y_2)$ is a constant of motion for the function $H(x_2,y_2)=\frac{1}{2}\exp(-2y_2)\left(y_2-x_2^2+\frac{1}{2}\right)=h$, until eventually $y_2=x_2^2$.
	Note that trajectories with $h\leq 0$ satisfy $y_2\leq x_2^2 - \frac{1}{2} <x_2^2$ at all times.
	Morover, solutions with $h<0$ correspond to unbounded solutions, which proves (iii). To see (v), note that the solution for $h=0$ is given by 
	$$
	\gamma_{c,2}(t_2)=(x_{c,2}(t_2),y_{c,2}(t_2)) = \left(\frac{1}{2}t_2,\frac{1}{4}t_2^2 - \frac{1}{2}\right),\qquad t_2\in \mathbb{R}.
	$$
	Moreover, $\frac{1}{4}$ is the global maximum of $H$ and $H(x_2,y_2)=h\in \left(0,\frac{1}{4}\right]$ is satisfied for $y_2=x_2^2=c\geq 0$ if and only of
	$$
	h= \frac{1}{2}\exp(-2c)\left(\frac{1}{2}\right) \Leftrightarrow c=c(h):= -\frac{1}{2}\log(4h) = -\log(2\sqrt{h}).
	$$
	Here, $h=\frac{1}{4}$ corresponds to the equilibrium $(0,0)$.
	Each trajectory with $h\in \left(0,\frac{1}{4}\right)$ and starting with $y_2< x_2^2$ approaches the point $(\sqrt{c(h)},c(h))$. 
	In fact, for $y_2< x_2^2$ and $x_2<0$, there hold $y_2'<0$ and $x_2'>0$, but since all constants of motion for $H$ with $h\in \left(0,\frac{1}{4}\right)$ correspond to periodic orbits, $x_2'=0=- y_2 + x_2^2$ must hold at a certain time. This condition is only satisfied for the points $(\pm \sqrt{c(h)}, c(h))$, but $(x_2^2,y_2)$ is moving away from the point $(-\sqrt{c(h)},c(h))$ if the initial condition satisfies $y_2< x_2^2$. This proves (ii), and (i) follows because $x_2'=0$, $y_2'<0$ hold for $y_2>x_2^2$, $x_2<0$. Statement~iv follows directly from \eqref{Eq:const_of_mo1}--\eqref{Eq:const_of_mo2}.
\end{proof}

With $\varepsilon=r_2^2, \lambda=r_2\lambda_2$, we can find $C_5>0$ such that the set
\begin{equation}\label{Def:U_2}
\begin{split}
U_{2}^{2}:= \{&(x_2,y_2+\tilde{y}_2,r_2,\lambda_2):\
x_2 - \lambda_2 + r_2(a_1x_2^2 + a_2y_2) = 0,\\
& \tilde{x}_2\in (-C_5(\varepsilon + 
|\lambda|), 	C_5(\varepsilon + 
|\lambda|)) \}\cap V_2
\end{split}
\index{zm@$U_{2}^{2}$| {\eqref{Def:U_2}, }}
\end{equation}
contains the set $\{g_2(x_2,y_2,r_2,\lambda_2)=0\}\cap V_2$ of roots of $g_2$, where 
$$g_2(x_2,y_2,r_2,\lambda_2)=x_2 - \lambda_2 + r_2(a_1x_2^2 + a_2y_2) + \mathcal{O}(r_2(|\lambda_2| + r_2))
$$ is the function on the right side of \eqref{Eq:blowup_y_2}.
The function $x_2\rightarrow - \frac{x_2- \lambda_2 + r_2a_1x_2^2}{a_2}$ is right-curved and decreasing for $x_2>-\frac{1}{2a_1r_2}$. For $r_2\in [-\rho,\rho]$ and $\rho$ small enough, $x_2>-\frac{1}{2a_1r_2}$ holds for all $(x_2,y_2)\in D$.
Hence, for fixed $r_2,\lambda_2$, the projection $P_{(x_2,y_2)}(U_{2}^{2})$ into $(x_2,y_2)$-space has right-curved and decreasing boundaries.
This implies that the projection $P_{(x_2,y_2)}(U_{2}^{2})$ of $U_{2}^{2}$ into $(x_2,y_2)$-space splits $\mathbb{R}^2$ into the lower set $P_{(x_2,y_2)}(U_{2}^{-})$, the separating set $P_{(x_2,y_2)}(U_{2}^{2})$ and the upper set $P_{(x_2,y_2)}(U_{2}^{+})$, where we define

\begin{equation}\label{Def:U_2^+-}
\begin{split}
U_{2}^{-}&:=\{g_2(x_2,y_2,r_2,\lambda_2)<0\}\backslash U_{2}^{2},\\ U_{2}^{+}&:=\{g_2(x_2,y_2,r_2,\lambda_2)>0\}\backslash U_{2}^{2}, \text{ and}\\
U_{0,2}^{-}&:=U_{2}^{-}\cap \mathcal{C}_0,\quad U_{0,2}^{+}:=U_{2}^{+}\cap \mathcal{C}_0.
\end{split}
\index{zn@$U_{2}^{-}, U_{2}^{+}, U_{0,2}^{-}, U_{0,2}^{+}$| {\eqref{Def:U_2^+-}, }}
\end{equation}

\subsection{Separation into Domains}\label{Sec:Domains}
With the notation of \eqref{Def:U_1^1},\eqref{Def:U_1^+-}, \eqref{Def:U_2} and \eqref{Def:U_2^+-} we define 
\begin{equation}\label{Def: U^1}
U^{1}:=P_{(x,y)}\Phi_1(U_{1}^{1}).\index{zo@$U^{1}$| {\eqref{Def: U^1} ,}}
\end{equation}
For $\lambda<\lambda_{H}$ we define 
\begin{equation}\label{Def:tildeU_2^2}
\begin{split}
\tilde{U}_{2}^{2}:=\{&(x_2,y_2+\tilde{y}_2,r_2,\lambda_2):\
x_2 - \lambda_2 + r_2(a_1x_2^2 + a_2y_2) = 0,\\
& \tilde{x}_2\in (-C_6(\sqrt{\varepsilon} + |\lambda|), 	C_5(\varepsilon + 
|\lambda|)) \}\cap V_2
\end{split}\index{zp@$\tilde{U}_{2}^{2}$| {\eqref{Def:tildeU_2^2}, }}
\end{equation}
for some appropriate $C_6>0$. For $\tilde{\Gamma}_{\lambda,\varepsilon}, p_-, p_+$ as introduced before Theorem~\ref{Thm:canard2}, we can find some $\lambda_*>\lambda_{H}$ such that $p_+\in P_{(x,y)}\Phi_2(U_{2}^{2})$ for $\lambda\in (\lambda_{H},\lambda_*)$ and $p_+\notin P_{(x,y)}\Phi_2(U_{2}^{2})$ for $\lambda\in [\lambda_*,\lambda_{sc}]$. We define
\index{zq@$U^{2}$| {\eqref{U^2} ,}}
\begin{equation}\label{U^2}
\begin{split}
U^{2}&:= P_{(x,y)}\Phi_2(\tilde{U}_{2}^{2}),
\quad \text{for }\lambda\in [-\lambda_0,\lambda_*),\\
U^{2}&:=
P_{(x,y)}\Phi_2(U_{2}^{2}),
\quad \text{for }\lambda\in [\lambda_*,\lambda_0].
\end{split}
\end{equation}
Finally we set
\begin{equation}\label{U^0}
U^{0}:=U^{1}\cup U^{2}.\index{zr@$U^0$| {\eqref{U^0}, }}
\end{equation}
Moreover, we set
$$U^{-}:= P_{(x,y)}(\Phi_1(U_{1}^{-}))\cup [(P_{(x,y)}\Phi_2(U_{2}^{-}))\backslash U^{2}],$$
$$U^{+}:= P_{(x,y)}(\Phi_1(U_{1}^{+}))\cup [(P_{(x,y)}\Phi_2(U_{2}^{+}))\backslash U^{2}],
$$
and 
$$U_0^{-}:= U^{-}\cap \mathcal{C}_0,\quad U_0^{+}:= U^{+}\cap \mathcal{C}_0.$$

\subsection{Proofs of the main results}

Before we prove the main results, we study the relative positioning of $\mathcal{C}_{\varepsilon}^a$, $\mathcal{C}_{\varepsilon}^r$ and $\mathcal{C}_0$.
\begin{lemma}\label{Lem:slow_flow}
	Consider the situation as in Theorems~\ref{Thm:canard1}--\ref{Thm:canard3}
	and fix $\varepsilon\in (0,\varepsilon_0]$.
	Then the following statements hold for $\lambda\in (-\lambda_0,\lambda_{sc}(\sqrt{\varepsilon}))$:
	\begin{enumerate}
		\item[(i)] Each trajectory starting in $U^{-}\backslash \mathcal{C}_0$ can only enter $\mathcal{C}_0$ in $U_0^{+}$.
		
		\item[(ii)] The slow flow corresponding to $\mathcal{C}_{\varepsilon}^a$ enters $V$ at $(p_{\varepsilon}^{a,x},p_{\varepsilon}^{a,y})\in \mathcal{C}_{\varepsilon}^a\cap\partial V$ with $p_{\varepsilon}^{a,x} < p_{e}^x$ and there holds $(p_{\varepsilon}^{a,x},p_{\varepsilon}^{a,y})\in \mathbb{R}^2\backslash\mathcal{C}_0$.
		The set $\mathcal{C}_{\varepsilon}^r\cap U^{-}$, where the corresponding slow flow leaves $V$, is
		located below $\mathcal{C}_{\varepsilon}^a$.
		
		\item[(iii)] The slow flow corresponding to $\mathcal{C}_{\varepsilon}^a$ enters $\mathcal{C}_0$ at $\mathcal{C}_{\varepsilon}^a\cap\mathcal{C}_{\partial}\in U_0^{+}$.
		For $\lambda\in (\lambda_{H}(\sqrt{\varepsilon}),\lambda_{sc}(\sqrt{\varepsilon}))$, this point is located to the right of $P_+$.
		
		\item[(iv)]
		$\mathcal{C}_{\varepsilon}^r\cap\mathcal{C}_0$ is either empty or contained in $U_0^{+}$ and to the right of $\mathcal{C}_{\varepsilon}^a$.
	\end{enumerate}
\end{lemma}

\begin{proof}
	(i)
	Let $(x(0),y(0))$ be located in $U^{-}\backslash \mathcal{C}_0$.
	Then by transformation to the chart $K_1$ or respectively $K_2$, $y$ is initially decreasing.
	Indeed, if $y(0)>0$ we may apply the transformation to chart $K_1$. By Lemma~\ref{Lem:chart1}, $x_1(0)=\frac{x(0)}{\sqrt{y(0)}}<-1$ and $r_1=\sqrt{y(0)}>0$ imply that $r_1$ is initially decreasing until the solution enters $V_{2,\rho^2}$, where we can change into chart $K_2$. Note that we still have $x_{1}<-1$, $r_{1}>0$ at this time.
	In chart $K_2$, we obtain $y_2<x_2^2$ and $y_2<-\frac{1}{r_2a_2}(x_2-\lambda_2 + r_2a_1x_2^2) + \mathcal{O}(|\lambda_2|+r_2)$.
	This implies that $x_2$ is increasing and $y_2$ decreasing. 
	Moreover, $U_{2}^{2}$ has right-curved boundaries.
	If $\lambda\in (\lambda_*,\lambda_{sc}(\sqrt{\varepsilon}))$ then	$p_+\notin U^{2}$, and the trajectory can only reach the parabola in $U^{+}$ since it is attracted by $\tilde{\Gamma}_{(\lambda,\varepsilon)}$.
	
	It remains to prove the statement for $\lambda\in [-\lambda_0,\lambda_*]$.
	We write $\mathcal{O}(2):=\mathcal{O}(r_2^2 + |r_2\lambda_2| + \lambda_2^2)$ and consider $\lambda\in [-\lambda_0,\lambda_*]$. 
	In order to prove the statement, it is enough to consider trajectories starting to the left of $U_2^2$ but close to $p_e$, and to prove that these trajectories reenter $\mathcal{C}_{0,2}$ to the right of $U_2^2$. Hence, we consider trajectories with initial conditions of the form $x_2(0) = \lambda_2 - c$, $y_2(0) = (\lambda_2 - c)^2$, where $c=\mathcal{O}(r_2 + |\lambda_2|)$ and $c>C_5(r_2^2 + r_2|\lambda_2|)$.
	We linearize \eqref{Eq:blowup_x_2}--\eqref{Eq:blowup_y_2} at $p_e=(\lambda_2,\lambda_2^2) + \mathcal{O}(2)$ to obtain
	\begin{align*}
	\left(
	\begin{matrix}
	x_2\\
	y_2
	\end{matrix}
	\right)' = 
	\left(
	\begin{matrix}
	2\lambda_2 	&& 	-1\\
	1 			&&	r_2a_2
	\end{matrix}
	\right)
	\left(
	\begin{matrix}
	x_2 - \lambda_2\\
	y_2 - \lambda_2^2
	\end{matrix}
	\right)
	+ \mathcal{O}(2, (x_2-\lambda_2)^2)
	=: A
	\left(
	\begin{matrix}
	x_2-\lambda_2\\
	y_2-\lambda_2^2
	\end{matrix}
	\right) + \mathcal{O}(2,(x_2-\lambda_2)^2).
	\end{align*}
	Note that the higher order term $\mathcal{O}((x_2-\lambda_2)^2)$ stems from the quadratic terms in $x_2$ in \eqref{Eq:blowup_x_2}--\eqref{Eq:blowup_y_2}.
	In order to estimate trajectories of the full system, we study the linearized system 
	\begin{align*}
	\left(
	\begin{matrix}
	x\\
	y
	\end{matrix}
	\right)' = 
	A
	\left(
	\begin{matrix}
	x\\
	y
	\end{matrix}
	\right).
	\end{align*}
	We denote by $(x,y)$ the solution of this system. Keep in mind that the term $(x_2(t)-\lambda_2)^2= x(t)^2$ has to be of order $\mathcal{O}(2)$ until the trajectory reenters $\mathcal{C}_0$, in order to ensure that $(x_2 , y_2)$ is approximated up to order $\mathcal{O}(2)$ by $(x + \lambda_2, y+\lambda_2^2)=(x,y) + p_e + \mathcal{O}(2)$.
	The eigenvalues of $A$ are given by $\mu=\lambda_2 + \frac{r_2a_2}{2} + \txti \frac{1}{2}\sqrt{|(2\lambda_2 - r_2a_2)^2 - 4|}$ and $\bar{\mu}$.
	We abbreviate 
	$$k:=\frac{1}{2}\sqrt{|(2\lambda_2 - r_2a_2)^2 - 4|}.
	$$
	The corresponding eigenvectors are given by 
	\benn
	v=\left(\begin{matrix}
		\lambda_2 - \frac{r_2a_2}{2} + \txti k\\
		1
	\end{matrix}\right) 
	\eenn
	and $\bar{v}$ respectively. Hence, we obtain real valued solutions of the linearized system as linear combinations of
	$$
	\txte^{(\lambda_2+\frac{r_2a_2}{2})t} 
	\left(
	\begin{matrix}
	(\lambda_2-\frac{r_2a_2}{2}) \cos(kt) - k \sin(kt)\\
	\cos(kt)
	\end{matrix}
	\right),
	$$
	$$
	\txte^{(\lambda_2+\frac{r_2a_2}{2})t} 
	\left(
	\begin{matrix}
	(\lambda_2-\frac{r_2a_2}{2}) \sin(kt) + k \cos(kt)\\
	\sin(kt)
	\end{matrix}
	\right).
	$$
	We want to estimate trajectories in \eqref{Eq:blowup_x_2}--\eqref{Eq:blowup_y_2} with initial conditions of the form $x_2(0) = \lambda_2 - c$, $y_2(0) = (\lambda_2 - c)^2$ for some $c>0$. Hence, in the linearized system, we obtain initial conditions (up to order $\mathcal{O}(2)$) of the form $x(0)=-c$, $y(0)=-2\lambda_2c + c^2=:d$, and hence the solution
	
	\begin{align*}
	\left(
	\begin{matrix}
	x(t)\\
	y(t)
	\end{matrix}
	\right)
	=
	\txte^{(\lambda_2+\frac{r_2a_2}{2})t} 
	\left[
	\left(
	\begin{matrix}
	-c\\
	d
	\end{matrix}
	\right)
	\cos(kt)
	-	\left(
	\begin{matrix}
	k^2 d + (\lambda_2 - \frac{r_2a_2}{2})^2d + (\lambda_2 - \frac{r_2a_2}{2})c\\
	(\lambda_2 - \frac{r_2a_2}{2})d + c
	\end{matrix}
	\right)
	\frac{\sin(kt)}{k}
	\right].
	\end{align*}
	Note that the requirement $x(t)^2=\mathcal{O}(2)$ is indeed fulfilled for $c=\mathcal{O}(r_2 + |\lambda_2|)$ and $t$ bounded, since then $c^2 + |cd| + d^2=\mathcal{O}(2)$.
	Also observe that $\frac{\sqrt{3}}{2}< k<2$ for $\lambda_2\in [-\mu,\lambda_{*,2}]$, where $\mu$ is small and with 
	$$
	\lambda_{H,2}=-\frac{r_2a_2}{2} + \mathcal{O}(r_2^2)
	< \lambda_{*,2} < \mu.$$
	For $c=\mathcal{O}(r_2 + |\lambda_2|)$ we obtain $d=\mathcal{O}(2)$ and therefore
	\begin{align*}
	\left(
	\begin{matrix}
	x(t)\\
	y(t)
	\end{matrix}
	\right)
	=
	\txte^{(\lambda_2+\frac{r_2a_2}{2})t} 
	\left[
	\left(
	\begin{matrix}
	-c\\
	0
	\end{matrix}
	\right)
	\cos(kt)
	-	\left(
	\begin{matrix}
	0 \\
	c
	\end{matrix}
	\right)
	\frac{\sin(kt)}{k}
	\right] + \mathcal{O}(2).
	\end{align*}
	At this stage, in order to justify that $(x_2,y_2)$ is indeed approximated up to order $\mathcal{O}(2)$ by $(\lambda_2 + x, \lambda_2^2 + y)$, it is crucial that $x'$ and $y'$ are not of order $\mathcal{O}(2)$ at the same time. This is satisfied if we require $c\neq \mathcal{O}(2)$. In this case, we obtain that - up to order $\mathcal{O}(2)$ - $(x_2,y_2)$ reaches the parabola $\mathcal{C}_{\partial,2}$ at the first time $t>0$ for which
	$$
	\mathcal{O}(2) = -2\lambda_2c\txte^{(\lambda_2+\frac{r_2a_2}{2})t}\cos(kt) + c^2 \txte^{2(\lambda_2+\frac{r_2a_2}{2})t}\cos^2(kt) = -\frac{c}{k}\txte^{(\lambda_2+\frac{r_2a_2}{2})t}\sin(kt).
	$$
	Since $\txte^{(\lambda_2+\frac{r_2a_2}{2})t}=1+ (\lambda_2+\frac{r_2a_2}{2})t + \mathcal{O}(2)$, we obtain 
	$$
	-\frac{c}{k}\txte^{(\lambda_2+\frac{r_2a_2}{2})t}\sin(kt)= -\frac{c}{k}\sin(kt) + \mathcal{O}(2).
	$$
	Hence, this time is approximately given by $t=\frac{\pi}{k}$.
	But since 
	$$(x(\pi/k),y(\pi/k))= (c,0) + \mathcal{O}(2) = (c,c^2) + \mathcal{O}(2),$$ we obtain that 
	$(x_2,y_2)$ reaches the parabola at some point 
	$$(\lambda_2 + c, (\lambda_2+c)^2) + \mathcal{O}(2).$$
	If $C_7(r_2 + |\lambda_2|) \geq c \geq C_6(r_2 + r_2|\lambda_2|)>C_5(r_2^2 + r_2|\lambda_2|)$ (as in \eqref{Def:tildeU_2^2}) with appropriate $C_6,C_7>0$, then $c=\mathcal{O}(r_2 + |\lambda_2|)$ as required and at the same time $c\neq \mathcal{O}(2)$ as required. Moreover, the point $(c,c^2) + \mathcal{O}(2)$ is located in $U_{2}^+$, which proves the statement for $\lambda\in [-\lambda_0,\lambda_*]$.
	
	(ii) We will show in Lemma~\ref{Lem:Slow_man} that $\mathcal{C}_{\varepsilon}^a\cap (V\backslash V_{2,\varepsilon})$ 
	is located below $\mathcal{C}_0$, provided that $\rho,\lambda_0$ are small enough. Moreover, the slow flow corresponding to $\mathcal{C}_{\varepsilon}^a$ enters $V$ at $(p_{\varepsilon}^{a,x},p_{\varepsilon}^{a,y})\in \mathcal{C}_{\varepsilon}^a\cap\partial V$ with $p_{\varepsilon}^{a,x} < p_{e}^x$. That $\mathcal{C}_{\varepsilon}^r\cap U^{-}$ is located below $\mathcal{C}_{\varepsilon}^a$ provided that $\lambda\in (-\lambda_0,\lambda_{sc}(\sqrt{\varepsilon}))$ is shown in \cite{KruSzm2}.
	
	(iii) By Theorem~\ref{Thm:class_canard1}, in the classical case, the slow flow corresponding to $\mathcal{C}_{\varepsilon}^a$ is attracted to $p_e$ for $\lambda\in (-\lambda_0,\lambda_{H}(\sqrt{\varepsilon})]$. For $\lambda\in (\lambda_{H}(\sqrt{\varepsilon}),\lambda_{sc}(\sqrt{\varepsilon}))$ it is attracted to $\Gamma_{\lambda,\varepsilon}$.
	Moreover, by (ii), $\mathcal{C}_{\varepsilon}^a$ can only be located outside of the periodic orbit.
	In both cases, this implies that the slow flow for $\mathcal{C}_{\varepsilon}^a$ enters $\mathcal{C}_0$ for the first time in $U_0^{+}$. The corresponding solution of \eqref{Eq:blowup_x}--\eqref{Eq:blowup_lamb} remains in $U_0^{+}$ until it leaves $V$.
	
	(iv) Follows from (ii) and (iii). 
\end{proof}
We are now able to prove the main results:
\begin{proof}[Proof of Theorem~\ref{Thm:canard1}]
	That $\Gamma_e$ exists follows from the implicit function theorem applied to \eqref{Eq:blowup_x_1}--\eqref{Eq:blowup_lamb_1} and \eqref{Eq:blowup_x_2}--\eqref{Eq:blowup_y_2}. The existence of the values $\lambda_{H}(\sqrt{\varepsilon})<\lambda_{sc}(\sqrt{\varepsilon})$ follows from Theorem~\ref{Thm:class_canard1}.
	By Lemma~\ref{Lem:chart1}, all trajectories starting in 
	$U_{1}^{-}$, see \eqref{Def:U_1^+-}, eventually reach the domain of definition of chart $K_2$, either above or below the critical manifold, i.e. with $x_1<-1$ close to $-1$ or still with $-1<x_1<0$, but to the left of $U^{1}$. Here, it is important to note that $U^{1}=P_{(x,y)}\Phi_1(U_1^1)$ is right-curved so that $(x_1,r_1,\varepsilon_1,\lambda_1)$ never enters $U_1^{1}$, see Lemma~\ref{Lem:chart1}.
	Note also that $U_{1}^{-}$ corresponds to the regions in (i) and (ii) of Lemma~\ref{Lem:chart1}.
	By Lemma~\ref{Lem:slow_flow}, the solution leaves the critical manifold at a point which is located above $\mathcal{C}_{\varepsilon}^a$, but reenters $\mathcal{C}_0$ in $U_0^{+}$. 
	Again by Lemma~\ref{Lem:slow_flow}, this behaviour applies for all $\lambda\in (-\lambda_0,\lambda_{sc}(\sqrt{\varepsilon}))$.
	
	A proof analog to that in \cite{KruSzm2} leads to the same results as in \cite[Theorem~3.1]{KruSzm2} about the dependence on $\lambda$ of the existence of a canard solution for the non-smooth problem \eqref{Eq:blowup_x}--\eqref{Eq:blowup_lamb}.
	In particular, we obtain the expansion 
	$$
	\lambda_c(\sqrt{\varepsilon}) = -\left(\frac{a_2}{2} + \frac{-2a_1 -2a_2}{8}\right)\varepsilon + \mathcal{O}(\varepsilon^{3/2})=\frac{a_1 - a_2}{4}\varepsilon + \mathcal{O}(\varepsilon^{3/2})
	$$
	for the critical value of $\lambda$ which yields maximal canard solutions. Here, it is important to note that each maximal canard solution, restricted to $V$, is strictly separated from $\mathcal{C}_0$ provided $\rho,\lambda_0$ are small enough.
	In particular, Lemma~\ref{Lem:Slow_man} shows that $\mathcal{C}_{\varepsilon}^a,\mathcal{C}_{\varepsilon}^r$ are located below $\mathcal{C}_0$ outside of $V_{2,\varepsilon}$. But inside $V_{2,\varepsilon}$, $\gamma_{c,2}$ lies strictly below $\mathcal{C}_{0,2}$, since $x_{c,2}^2 - y_{c,2} = \frac{1}{2}$. Because maximal canard solutions restricted to $V_{2,\varepsilon}$ are perturbations of $\gamma_{c,2}$ of order $\mathcal{O}(r_2,\lambda_2,r_2(|\lambda_2|+ r_2))$, and since $r_2 = \sqrt{\varepsilon}\in [0,\rho)$, $\lambda_2\in (-\mu,\mu)$, also this perturbation lies strictly below $\mathcal{C}_{0,2}$ for $\rho,\mu$ small.
	Finally, because maximal canard solutions are just the patching of $\mathcal{C}_{\varepsilon}^a,\mathcal{C}_{\varepsilon}^r$ restricted to $V\backslash V_{2,\varepsilon}$ and the perturbation of $\gamma_{c,2}$, this proves their existence and that their restriction to $V$ is located below $\mathcal{C}_0$.
	In particular, for $\lambda = \lambda_c(\sqrt{\varepsilon})$, the attracting slow manifold $\mathcal{C}_{\varepsilon}^a$ connects to $\mathcal{C}_{\varepsilon}^r$, and both are strictly separated from and located below $\mathcal{C}_0$ in $V$.
\end{proof}

We prove the second main result:
\begin{proof}[Proof of Theorem~\ref{Thm:canard2}]
	(i) follows from Theorem~\ref{Thm:canard1}.
	
	(ii) The first part of the statement follows from Theorem~\ref{Thm:class_canard1}.
	Let $\lambda\in (\lambda_{H}(\sqrt{\varepsilon}),\lambda_{sc}(\sqrt{\varepsilon}))$.
	As in the proof of Theorem~\ref{Thm:canard1} consider any solution starting in $U_0^{-}$ and to the left of $P_-$. This solution arrives at $\mathcal{C}_{\partial}$ to the left of $P_-$ and, by Lemma~\ref{Lem:slow_flow}, above $\mathcal{C}_{\varepsilon}^a$.
	Since the trajectory behaves as a classical canard solution while contained in $\mathbb{R}^2\backslash \mathcal{C}_0$, it remains between $\Gamma_{(\lambda,\varepsilon)}$ and $\mathcal{C}_{\varepsilon}^a$ until it reenters $\mathcal{C}_0$ to the right of $P_+$.
	That the solution reenters $\mathcal{C}_0$ at all follows from Lemma~\ref{Lem:slow_flow}.
	Now consider any solution starting in $V\backslash U^{0}$, between $P_-, P_+$ and above $\tilde{\Gamma}_{(\lambda,\varepsilon)}$.
	If the solution starts in $U_0^{-}$, then we can either apply the transformation to $U_{1}^{-}$, i.e. to chart $K_1$, and use Lemma~\ref{Lem:chart1}, or the transformation to chart $K_2$ and \eqref{Eq:blowup_x_2}--\eqref{Eq:blowup_y_2}. In particular, the solution moves vertically downwards until it arrives at $\mathcal{C}_{\partial}$, still between $P_-, P_+$ and above $\tilde{\Gamma}_{(\lambda,\varepsilon)}$.
	Further on, the trajectory behaves as a classical canard solution and is - by Theorem~\ref{Thm:class_canard2} - attracted by $\tilde{\Gamma}_{(\lambda,\varepsilon)}$ from the interior, until it reenters $\mathcal{C}_{\partial}$ in $U_0^{+}$ but to the left of $P_+$.
	
	(iii) As for $\lambda\in (-\lambda_0,\lambda_{sc}(\sqrt{\varepsilon}))$, one shows that
	$\Gamma_e$ is unstable. Since there are no further equilibria contained in $V$, and because $U^{0}$ is right-curved, this proves the statement.
	
	(iv) Follows from the definitions in Section~\ref{Sec:Domains}.
\end{proof}

\begin{proof}[Proof of Theorem~\ref{Thm:canard3}]
	For $\lambda\in (\lambda_{c}(\sqrt{\varepsilon}),\lambda_0]$, considering the relative position of the attracting and the repelling slow manifold for the classical system \eqref{Eq:class_blowup_x}--\eqref{Eq:class_blowup_lamb} implies that $\mathcal{C}_{\varepsilon}^a$ does not intersect $\mathcal{C}_{\partial}$ in $V$ at all, see \cite{KruSzm2,KruSzm3} and Lemma~\ref{Lem:slow_flow} and the part of the proof of Theorem~\ref{Thm:canard1} about maximal canard solutions. In particular, $\mathcal{C}_{\varepsilon}^a\subset V\backslash\mathcal{C}_0$.
	Moreover, the repelling slow manifold in the classical case (in backward time) spirals around $p_e$ until it eventually leaves $V$.
	For \eqref{Eq:blowup_x}--\eqref{Eq:blowup_lamb}, this implies that there exists a unique point $(p_c^x,p_c^y)\in U_0^{-}\cap\mathcal{C}_{\partial}$ together with its vertical extension $P_c=\{(p_c^x,y):\, y\in \mathbb{R}\}$ such that the following holds true: 
	
	The solution starting at $(p_c^x,p_c^y)$ does not reenter $\mathcal{C}_0$ again. In particular, it leaves $V$ in $V\backslash\mathcal{C}_0$.
	Moreover, all solutions starting in $U_0^{-}\cap\mathcal{C}_{\partial}$ and to the right of $P_c$ reenter $\mathcal{C}_0$ in $U_0^{+}$, see also Lemma~\ref{Lem:slow_flow}.
	
	This implies that all trajectories starting in $V\backslash U^{0}$ and to the left of $P_c$ also leave $V$ in $V\backslash\mathcal{C}_0$, since those trajectories stay to the left and below the solution starting at $(p_c^x,p_c^y)$.
	
	Moreover, any solution starting in $U_0^{+}$ remains in this set until it leaves $V$, and any solution starting in $U_0^{-}$ and to the right of $P_c$ reaches $\mathcal{C}_{\partial}$ to the right of $P_c$, so that it reenters $\mathcal{C}_0$ in $U_0^{+}$, see Figure~\ref{Fig:LargeLambda}.
\end{proof}

\subsection{Slow manifold}
In order to determine the dynamics of trajectories for \eqref{Eq:blowup_x}--\eqref{Eq:blowup_lamb}, it is useful to know where there exist branches of the slow manifold which are separated from the critical manifold. That is, we want to know whether and where the attracting/repelling slow manifold of the smooth system \eqref{Eq:class_blowup_x}--\eqref{Eq:class_blowup_lamb} is located below the parabola $\mathcal{C}_{\partial}=\{(x,y):\, y=x^2\}$.
In particular, those branches $\mathcal{C}_\varepsilon^a$ and $\mathcal{C}_\varepsilon^r$ are still observed in system \eqref{Eq:blowup_x}--\eqref{Eq:blowup_lamb}. However, also each trajectory in $\mathcal{C}_0$ follows the slow subsystem. This provides a key difference between system \eqref{Eq:class_blowup_x}--\eqref{Eq:class_blowup_lamb} with one-dimensional critical manifold $\mathcal{C}_\partial$ and system \eqref{Eq:blowup_x}--\eqref{Eq:blowup_lamb} with critical manifold $\mathcal{C}_0$ of codimension zero.
In this section, we study $\mathcal{C}_\varepsilon^a$ and $\mathcal{C}_\varepsilon^r$ outside the domain of $K_2$, i.e. in sets of the form $\{(x,y):\, y\in (\varepsilon,\rho^2),\ x\in (-x_{1,0}\sqrt{y},x_{1,0}\sqrt{y})\}$, where Fenichel theory and asymptotic expansion techniques are applicable.
In particular, we expand $x=x_{(0)} + \varepsilon x_{(1)} + \mathcal{O}(\varepsilon^2)$ and write $y=y(x)$.
Then, we obtain due to \eqref{Eq:class_blowup_x}--\eqref{Eq:class_blowup_lamb} and the chain rule
\begin{align*}
\left(\frac{\txtd x_{(0)}}{\txtd y} + \varepsilon\frac{\txtd x_{(1)}}{\txtd y} + \mathcal{O}(\varepsilon^2)\right)\varepsilon g 
&= \frac{\txtd x}{\txtd y}\varepsilon g = \frac{\txtd x}{\txtd y}y'\\
= &x'
= - y + x^2 = -y + x_{(0)}^2 + \varepsilon 2x_{(0)}x_{(1)} + \mathcal{O}(\varepsilon^2).
\end{align*}
Recall that $g=xg_1- \lambda g_2 + yg_3$ is $\mathrm{C}^r$-smooth for $r\geq 3$ and $g_i=g_i(x,y,\lambda)$ for $i=1,2,3$, with $g_1,g_2 = 1+\mathcal{O}(x,y,\lambda)$. 
Inserting the asymptotic expansion for $x$, we obtain by a Taylor expansion
\begin{align*}
g(x,y,\lambda)&= [g_1(0,0,0) + \partial_x g_1(0,0,0)(x_{(0)} + \varepsilon x_{(1)})](x_{(0)} + \varepsilon x_{(1)})\\
& - g_2(0,0,0)\lambda + g_3(0,0,0)y + \mathcal{O}(\varepsilon^2,xy,x\lambda,\lambda^2,y^2)\\
&= [1 + a_1(x_{(0)} + \varepsilon x_{(1)})](x_{(0)} + \varepsilon x_{(1)})\\
& - \lambda + a_2y + \mathcal{O}(\varepsilon^2,xy,x\lambda,\lambda^2,y^2)\\
&= x_{(0)} + a_1 x_{(0)}^2 - \lambda + a_2y + \mathcal{O}(\varepsilon,xy,x\lambda,\lambda^2,y^2).
\end{align*}
Inserting this into the previous equation, sorting for terms of order $\mathcal{O}(1)$ and $\mathcal{O}(\varepsilon)$, and neglecting terms of order $\mathcal{O}(xy,x\lambda,\lambda^2,y^2)$ implies the conditions
$$
y=x_{(0)}^2 \Leftrightarrow x_{(0)}=\pm\sqrt{y},\qquad \text{and then}
$$
$$
\pm\frac{1}{2\sqrt{y}}(\pm\sqrt{y} + (a_1+a_2)y - \lambda)  = \pm 2\sqrt{y}x_{(1)}.
$$
Note here that $y>\varepsilon>0$. Solving for $x_{(1)}$ yields
$$
x_{(1)} = \frac{1}{4y}(\pm\sqrt{y} + (a_1+a_2)y - \lambda).
$$
Therefore,
\begin{align*}
x&= x_{(0)} + \varepsilon x_{(1)} + \mathcal{O}(\varepsilon^2)
= \left\lbrace\begin{matrix}
-\sqrt{y} + \varepsilon\frac{1}{4y}(-\sqrt{y} + (a_1+a_2)y - \lambda)\\
\sqrt{y} + \varepsilon\frac{1}{4y}(\sqrt{y} + (a_1+a_2)y - \lambda)
\end{matrix}\right.
+\mathcal{O}(\varepsilon^2).
\end{align*}
From this we obtain
\begin{align*}
x^2&= \left\lbrace\begin{matrix}
y - \varepsilon\frac{1}{2\sqrt{y}}(-\sqrt{y} + (a_1+a_2)y - \lambda)\\
y + \varepsilon\frac{1}{2\sqrt{y}}(\sqrt{y} + (a_1+a_2)y - \lambda)
\end{matrix}\right.
+\mathcal{O}(\varepsilon^2)\\
&= \left\lbrace\begin{matrix}
y + \frac{\varepsilon}{2}\left(1 - (a_1+a_2)\sqrt{y} + \frac{\lambda}{\sqrt{y}}\right)\\
y + \frac{\varepsilon}{2}\left(1 + (a_1+a_2)\sqrt{y} - \frac{\lambda}{\sqrt{y}}\right)
\end{matrix}\right.
+\mathcal{O}(\varepsilon^2).
\end{align*}
In order to obtain branches of the slow manifold which are located below $\mathcal{C}_{\partial}$, we have to check if the conditions
\begin{align*}
C < \left\lbrace\begin{matrix}
1 - (a_1+a_2)\sqrt{y} + \frac{\lambda}{\sqrt{y}},\\
1 + (a_1+a_2)\sqrt{y} - \frac{\lambda}{\sqrt{y}},
\end{matrix}\right.
\end{align*}
hold for some $C>0$, since this implies that the $\mathcal{O}(\varepsilon)$-terms are strictly positive and exceed the $\mathcal{O}(\varepsilon^2)$-terms, provided that $\varepsilon\in (0,\rho^2]$ with $\rho$ small enough.
The first expression can be estimated from below by
$$
1 - (a_1+a_2)\sqrt{y} + \frac{\lambda}{\sqrt{y}} > 1 - (a_1+a_2)\rho + \frac{\lambda}{\rho},
$$
for $\varepsilon\in (0,\rho^2]$ and the right side is strictly larger than some $C>0$ for $\lambda\in (-\lambda_0,\lambda_{c}]$, if $\rho,\lambda_0>0$ are chosen small enough.
Similarly, the second expression can be estimated from below by
$$
1 + (a_1+a_2)\sqrt{y} - \frac{\lambda}{\sqrt{y}} > 1 + (a_1+a_2)\sqrt{\varepsilon} - \frac{\lambda}{\sqrt{\varepsilon}},
$$
and also this time the right side is strictly larger than some $C>0$ for $\varepsilon\in (0,\rho^2]$ and $\lambda\in (-\lambda_0,\lambda_{c}]$, if $\rho,\lambda_0>0$ are chosen small enough. Note here, that $\lambda_{c}=\frac{a_1-a_2}{4}\varepsilon + \mathcal{O}(\varepsilon^{3/2})$.
In particular, this proves the following lemma.
\begin{lemma}\label{Lem:Slow_man}
	If $\rho,\lambda_0>0$ are chosen small enough, then for all $\varepsilon\in (0,\rho^2]$ and $\lambda\in (-\lambda_0,\lambda_{c}]$, the restrictions to $V\backslash V_{2,\varepsilon}$ of the attracting branch $\mathcal{C}_{\varepsilon}^a$ and the repelling branch $\mathcal{C}_{\varepsilon}^r$ of the slow manifold for system \eqref{Eq:class_blowup_x}--\eqref{Eq:class_blowup_lamb} are both located below the critical manifold $\mathcal{C}_{\partial}$. Here, $V=V_{\rho^2}$ is defined according to Definition~\ref{Def:nbhd}.
\end{lemma}

\textbf{Acknowledgments:} This research was supported by the DFG Collaborative Research Center TRR109, Discretization in Geometry and Dynamics. CK would also like to thank the VolkswagenStiftung for support via a Lichtenberg Professorship.


\begin{thebibliography}{10}
	
	\bibitem{BenoitCallotDienerDiener}
	E.~Beno\^{i}t, J.L. Callot, F.~Diener, and M.~Diener.
	\newblock Chasse au canards.
	\newblock {\em Collect. Math.}, 31:37--119, 1981.
	
	\bibitem{BrokateSprekels}
	M.~Brokate and J.~Sprekels.
	\newblock {\em Hysteresis and Phase Transitions}.
	\newblock Springer, 1996.
	
	\bibitem{Desrochesetal1}
	M.~Desroches, E.~Freire, S.J. Hogan, E.~Ponce, and P.~Thota.
	\newblock Canards in piecewise-linear systems: explosions and super-explosions.
	\newblock {\em Proc. R. Soc. A}, 469:20120603, 2013.
	
	\bibitem{Dumortier1}
	F.~Dumortier.
	\newblock Techniques in the theory of local bifurcations: Blow-up, normal
	forms, nilpotent bifurcations, singular perturbations.
	\newblock In D.~Schlomiuk, editor, {\em Bifurcations and Periodic Orbits of
		Vector Fields}, pages 19--73. Kluwer, Dortrecht, The Netherlands, 1993.
	
	\bibitem{DumortierRoussarie}
	F.~Dumortier and R.~Roussarie.
	\newblock {\em Canard Cycles and Center Manifolds}, volume 121 of {\em Memoirs
		Amer. Math. Soc.}
	\newblock AMS, 1996.
	
	\bibitem{Fenichel4}
	N.~Fenichel.
	\newblock Geometric singular perturbation theory for ordinary differential
	equations.
	\newblock {\em J. Differential Equat.}, 31:53--98, 1979.
	
	\bibitem{Jones}
	C.K.R.T. Jones.
	\newblock Geometric singular perturbation theory.
	\newblock In {\em Dynamical Systems (Montecatini Terme, 1994)}, volume 1609 of
	{\em Lect. Notes Math.}, pages 44--118. Springer, 1995.
	
	\bibitem{KruSzm3}
	M.~Krupa and P.~Szmolyan.
	\newblock Extending geometric singular perturbation theory to nonhyperbolic
	points - fold and canard points in two dimensions.
	\newblock {\em SIAM J. Math. Anal.}, 33(2):286--314, 2001.
	
	\bibitem{KruSzm2}
	M.~Krupa and P.~Szmolyan.
	\newblock Relaxation oscillation and canard explosion.
	\newblock {\em J. Differential Equat.}, 174:312--368, 2001.
	
	\bibitem{KuehnScaleSN}
	C.~Kuehn.
	\newblock {Scaling of saddle-node bifurcations: degeneracies and rapid
		quantitative changes}.
	\newblock {\em J. Phys. A: Math. and Theor.}, 42(4):045101, 2009.
	
	\bibitem{KuehnBook}
	C.~Kuehn.
	\newblock {\em Multiple Time Scale Dynamics}.
	\newblock Springer, 2015.
	\newblock 814 pp.
	
	\bibitem{KuehnMuench}
	C.~Kuehn and C.~M{\"u}nch.
	\newblock Generalized play hysteresis operators as limits of fast-slow systems.
	\newblock {\em SIAM J. Appl. Dyn. Syst.}, 16(3):1650--1685, 2017.
	
	\bibitem{Mielke3}
	A.~Mielke.
	\newblock Evolution of rate-independent systems.
	\newblock In C.M. Dafermos and E.~Feireisl, editors, {\em Handbook of
		Differential Equations Evolutionary Equations}, volume~2, pages 461--559.
	North-Holland, 2005.
	
	\bibitem{MielkeRoubicek}
	A.~Mielke and T.~Roub{\'i}{\v{c}}ek.
	\newblock {\em Rate-independent Systems}.
	\newblock Springer, 2015.
	
	\bibitem{RobertsGlendinning}
	A.~Roberts and P.~Glendinning.
	\newblock Canard-like phenomena in piecewise-smooth {Van der Pol} systems.
	\newblock {\em Chaos}, 24(2):023138, 2014.
	
	\bibitem{Visintin}
	A.~Visintin.
	\newblock {\em Differential Models of Hysteresis}.
	\newblock Springer, 1994.
	
	\bibitem{WigginsIM}
	S.~Wiggins.
	\newblock {\em Normally Hyperbolic Invariant Manifolds in Dynamical Systems}.
	\newblock Springer, 1994.
	
\end{thebibliography}
\end{document}